\documentclass[12pt]{article}

\usepackage[left=1in,top=1in,right=1in,bottom=1in]{geometry}

\usepackage{setspace}

\usepackage[authoryear,round]{natbib}

\usepackage{xcolor}
\definecolor{Bleu}{RGB}{0,0,204}
\usepackage{hyperref}
\hypersetup{
colorlinks,
  citecolor=Bleu,
  linkcolor=Bleu,
  urlcolor=Bleu,
hypertexnames=false} 

\usepackage{graphicx}
\usepackage[textwidth=8em,textsize=small]{todonotes}

\usepackage{amsmath,amsthm,amssymb}
\usepackage{bm}

\newcommand{\secref}[1]{\hyperref[#1]{Section~\ref{#1}}}
\newcommand{\condref}[1]{\hyperref[#1]{Condition~\ref{#1}}}
\newcommand{\figref}[1]{\hyperref[#1]{Figure~\ref{#1}}}
\newcommand{\tabref}[1]{\hyperref[#1]{Table~\ref{#1}}}

\usepackage{amssymb}
\usepackage{graphicx}
\usepackage{latexsym}
\usepackage{amsmath}
\usepackage{enumitem}

\newtheorem{theorem}{Theorem}
\newtheorem{corollary}{Corollary}
\newtheorem{lemma}{Lemma}
\newtheorem{remark}{Remark}

\newcommand{\openr}{\hbox{${\rm I\kern-.2em R}$}}
\newcommand{\openn}{\hbox{${\rm I\kern-.2em N}$}}

\newcommand{\Rem}{\operatorname{Rem}}

\newcommand{\norm}[1]{\lVert#1\|}

\newcommand{\onenorm}[1]{\lVert#1\|_{1,P_0}}
\newcommand{\twonorm}[1]{\lVert#1\|_{2,P_0}}
\newcommand{\fournorm}[1]{\lVert#1\|_{4,P_0}}
\def\eqd{\,{\buildrel d \over =}\,} 
 
\newcommand{\Un}{\mathbb{U}_n}
\newcommand{\Id}{\textnormal{Id}}
\newcommand{\expit}{\operatorname{expit}}

\usepackage[mathscr]{euscript}

\allowdisplaybreaks

\usepackage{authblk}

\title{An Omnibus Nonparametric Test of Equality in Distribution for Unknown Functions}
\author[1,2]{Alexander R.  Luedtke}

\affil[1]{\footnotesize   Vaccine  and   Infectious  Disease   Division,  Fred
  Hutchinson Cancer Research Center, USA}
\affil[2]{\footnotesize   Public Health Sciences  Division,  Fred
  Hutchinson Cancer Research Center, USA}

\author[3]{Mark J. van der Laan}
\affil[3]{\footnotesize   Division of Biostatistics, University of California, Berkeley, USA} 

\author[4,1]{Marco Carone}
\affil[4]{\footnotesize   Department of Biostatistics, University of Washington, USA}

\date{}

\begin{document}
\maketitle

\begin{abstract}
We present a novel family of nonparametric omnibus tests of the hypothesis that two unknown but estimable functions are equal in distribution when applied to the observed data structure. We developed these tests, which represent a generalization of the maximum mean discrepancy tests described in \cite{Grettonetal2006},  using recent developments from the higher-order pathwise differentiability literature. Despite their complex derivation, the associated test statistics can be expressed rather simply as U-statistics. We study the asymptotic behavior of the proposed tests under the null hypothesis and under both fixed and local alternatives. We provide examples to which our tests can be applied and show that they perform well in a simulation study. As an important special case, our proposed tests can be used to determine whether an unknown function, such as the conditional average treatment effect, is equal to zero almost surely.
\end{abstract}
{\small\textbf{Kewyords:} higher order pathwise differentiability, maximum mean discrepancy, omnibus test, equality in distribution, infinite dimensional parameter.}

\section{Introduction}

In many scientific problems, it is of interest to determine whether two particular functions are equal to each other. In many settings these functions are unknown and may be viewed as features of a data-generating mechanism from which observations can be collected. As such, these functions can be learned from available data, and estimates of these respective functions can then be compared. To reduce the risk of deriving misleading conclusions due to model misspecification, it is appealing to employ flexible statistical learning tools to estimate the unknown functions. Unfortunately, inference is usually extremely difficult when such techniques are used, because the resulting estimators tend to be highly irregular. In such cases, conventional techniques for constructing confidence intervals or computing p-values are generally invalid, and a more careful construction, as exemplified by the work presented in this article, is required. 

To formulate the problem statistically, suppose that $n$ independent observations $O_1,\ldots,O_n$ are drawn from a distribution $P_0$ known only to lie in the nonparametric statistical model, denoted by $\mathscr{M}$. Let $\mathscr{O}$ denote  the support of $P_0$, and suppose that $P\mapsto R_P$ and $P\mapsto S_P$ are parameters mapping from $\mathscr{M}$ onto the space of univariate bounded real-valued measurable functions defined on $\mathscr{O}$, i.e. $R_P$ and $S_P$ are elements of the space of univariate bounded real-valued measurable functions defined on $\mathscr{O}$. For brevity, we will write $R_0\triangleq R_{P_0}$ and $S_0\triangleq S_{P_0}$. Our objective is to test the null hypothesis \[\mathscr{H}_0: R_0(O)\stackrel{d}{=}S_0(O)\] versus the complementary alternative $\mathscr{H}_1: \mbox{not }\mathscr{H}_0$, where $O$ follows the distribution $P_0$ and the symbol $\eqd$ denotes equality in distribution.  We note that $R_0(O)\eqd S_0(O)$ if $R_0\equiv S_0$, i.e. $R_0(O)=S_0(O)$ almost surely, but not conversely. The case where $S_0\equiv 0$ is of particular interest since then the null simplifies to $\mathscr{H}_0 : R_0\equiv 0$. Because $P_0$ is unknown, $R_0$ and $S_0$ are not readily available. Nevertheless, the observed data can be used to estimate $P_0$ and hence each of $R_0$ and $S_0$. The approach we propose will apply to functionals within a specified class described later.

Before presenting our general approach, we describe some motivating examples. Consider the data structure $O=(W,A,Y)\sim P$, where $W$ is a collection of covariates, $A$ is binary treatment indicator, and $Y$ is a bounded outcome, i.e., there exists a universal $c$ such that, for all $P\in\mathcal{M}$, $P(|Y|\le c) = 1$. Note that, in our examples, the condition that $Y$ is bounded cannot easily be relaxed, as the parameter from \cite{Grettonetal2006} on which we will base our testing procedure requires that the quantities under consideration have compact support. \begin{enumerate}
\item[] \hspace{-.2in}\textbf{\hypertarget{ex:1}{Example 1}:} Testing a null conditional average treatment effect.

If $R_P(o)\triangleq E_P\left(Y\mid A=1,W=w\right)-E_P\left(Y\mid A=0,W=w\right)$ and $S_P\equiv 0$, the null hypothesis corresponds to the absence of a conditional average treatment effect. This definition of $R_P$ corresponds to the so-called blip function introduced by \cite{Robins04}, which plays a critical role in defining optimal personalized treatment strategies \citep{Chakraborty&Moodie2013}.

\item[] \hspace{-.2in}\textbf{\hypertarget{ex:2}{Example 2}:} Testing for equality in distribution of regression functions in two populations.

Suppose the setting of the previous example, but where $A$ represents membership to population $0$ or $1$. If $R_P(o)\triangleq E_P\left(Y\mid A=1,W=w\right)$ and $S_P(o)\triangleq E_P\left(Y\mid A=0,W=w\right)$, the null hypothesis corresponds to the outcome having the conditional mean functions, applied to a random draw of the covariate, having the same distribution in these two populations. We note here that our formulation considers selection of individuals from either population as random rather than fixed so that population-specific sample sizes (as opposed to the total sample size) are themselves random. The same interpretation could also be used for the previous example, now testing if the two regression functions are equivalent.

\item[] \hspace{-.2in}\textbf{\hypertarget{ex:3}{Example 3}:} Testing a null covariate effect on average response.

Suppose now that the data unit only consists of $O\triangleq (W,Y)$. If $R_P(o)\triangleq E_P\left(Y\mid W=w\right)$ and $S_P\equiv 0$, the null hypothesis corresponds to the outcome $Y$ having conditional mean zero in all strata of covariates. This may be interesting when zero has a special importance for the outcome, such as when the outcome is the profit over some period.

\item[] \hspace{-.2in}\textbf{\hypertarget{ex:4}{Example 4}:} Testing a null variable importance.

Suppose again that $O\triangleq (W,Y)$ and $W\triangleq (W(1),W(2),\ldots,W(K))$. Denote by $W(-k)$ the vector $(W(i):1\leq i\leq K, i\neq k)$. Setting $R_P(o)\triangleq E_P\left(Y\mid W=w\right)$ and $S_P(o)\triangleq E_P\left(Y\mid W(-k)=w(-k)\right)$, the null hypothesis corresponds to $W(k)$ having null variable importance in the presence of $W(-k)$ with respect to the conditional mean of $Y$ given $W$ in the sense that $E_{P}\left(Y\mid W\right)=E_P\left(Y\mid W(-k)\right)$ almost surely. This is true because if $R_0(W)\eqd S_0(W(-k))$, the latter random variables have equal variance and so \begin{align*}
E_{P_0}&\left\{Var_{P_0}\left[R_0(W)\mid W(-k)\right]\right\}\ \\
&=\ Var_{P_0}\left[R_0(W)\right]-Var_{P_0}\left\{E_{P_0}\left[R_0(W)\mid W(-k)\right]\right\}\\
&=\ Var_{P_0}\left[R_0(W)\right]-Var_{P_0}\left[S_0(W(-k))\right]\ =\ 0\ ,
\end{align*} implying that $Var_{P_0}\left[R_0(W)\mid W(-k)\right]=0$ almost surely. Thus, a test of $R_P(O)\eqd S_P(O)$ is equivalent to a test of almost sure equality between $R_P$ and $S_P$ in this example. We will show in \secref{sec:illustr} that our approach cannot be directly applied to this example, but that a simple extension yields a valid test.
\end{enumerate}

\cite{Grettonetal2006} investigated the related problem of testing equality between two distributions in a two-sample problem. They proposed estimating the maximum mean discrepancy (hereafter referred to as MMD), a non-negative numeric summary that equals zero if and only if the two distributions are equal.  They also  investigated related problems using this technique \citep[see, e.g.,][]{Grettonetal2009,Grettonetal2012,Sejdinovicetal2013}. In this work, we also utilize the MMD as a parsimonious summary of equality but consider the more general problem wherein the null hypothesis relies on unknown functions $R_0$ and $S_0$ indexed by the data-generating distribution $P_0$.

Other investigators have proposed omnibus tests of hypotheses of the form $\mathscr{H}_0$ versus $\mathscr{H}_1$ in the literature. In the setting of \hyperlink{ex:1}{Example 1} above, the work presented in \cite{Racineetal2006} and \cite{Lavergneetal2015} is particularly relevant. The null hypothesis of interest in these papers consists of the equality $E_{P_0}\left(Y\mid A,W\right)=E_{P_0}\left(Y\mid W\right)$ holding almost surely. If individuals have a nontrivial probability of receiving treatment in all strata of covariates, this null hypothesis is equivalent to $\mathscr{H}_0$. In both these papers, kernel smoothing is used to estimate the required regression functions. Therefore, key smoothness assumptions are needed for their methods to yield valid conclusions. The method we present does not hinge on any particular class of estimators and therefore does not rely on this condition.

To develop our approach, we use techniques from the higher-order pathwise differentiability literature \citep[see, e.g.,][]{Pfanzagl1985,RobinsetTchetgenvanderVaart08,vanderVaart2014,Caroneetal2014}. Despite the elegance of the theory presented by these various authors, it has been unclear whether these higher-order methods are truly useful in infinite-dimensional models since most functionals of interest fail to be even second-order pathwise differentiable in such models. This is especially troublesome in problems in which under the null the first-order derivative of the parameter of interest (in an appropriately defined sense) vanishes, since then there seems to be no theoretical basis for adjusting parameter estimates to recover parametric rate asymptotic behavior. At first glance, the MMD parameter seems to provide one such disappointing example, since its first-order derivative indeed vanishes under the null. The latter fact is a common feature of problems wherein the null value of the parameter is on the boundary of the parameter space. It is also not an entirely surprising phenomenon, at least heuristically, since the MMD achieves its minimum of zero under the null hypothesis. Nevertheless,  we are able to show that this parameter is indeed second-order pathwise differentiable under the null hypothesis -- this is a rare finding in infinite-dimensional models. As such, we can employ techniques from the recent higher-order pathwise differentiability literature to tackle the problem at hand. To the best of our knowledge, this is the first instance in which these techniques are directly used (without any form of approximation) to resolve an open methodological problem.

This paper is organized as follows. In \secref{sec:rep}, we formally present our parameter of interest, the squared MMD between two unknown functions, and establish asymptotic representations for this parameter based on its higher-order differentiability, which, as we formally establish, holds even when the MMD involves estimation of unknown nuisance parameters. In \secref{sec:est}, we discuss estimation of this parameter, discuss the corresponding hypothesis test and study its asymptotic behavior under the null. We study the consistency of our proposed test under fixed and local alternatives in \secref{sec:local}. We revisit our examples in \secref{sec:illustr} and provide an additional example in which we can still make progress using our techniques even though our regularity conditions fails. In \secref{sec:num}, we present results from a simulation study to illustrate the finite-sample performance of our test, and we end with concluding remark in \secref{sec:concl}.

Appendix A reviews higher-order pathwise differentiability. Appendix B gives a summary of the empirical $U$-process results from \cite{Nolan&Pollard1988} that we build upon. All proofs can be found in Appendix C.

\section{Properties of maximum mean discrepancy}\label{sec:rep}

\subsection{Definition}

For a distribution $P$ and mappings $T$ and $U$, we define
\begin{align}
\Phi^{TU}(P)\triangleq  \iint e^{-\left[T_P(o_1)-U_P(o_2)\right]^2}dP(o_1)dP(o_2) \label{eq:Phidef}
\end{align} and set $\Psi(P)\triangleq  \Phi^{RR}(P) - 2\Phi^{RS}(P) + \Phi^{SS}(P)$.  The MMD between the distributions of $R_P(O)$ and $S_P(O)$ when $O\sim P$ is given by $\sqrt{\Psi(P)}$ and is always well-defined because $\Psi(P)$ is non-negative. 
Indeed, denoting by $\psi_0$ the true parameter value $\Psi(P_0)$, Theorem 3 of \cite{Grettonetal2006} establishes that $\psi_0$ equals zero if $\mathscr{H}_0$ holds and is otherwise strictly positive. Though the study in \cite{Grettonetal2006} is restricted to two-sample problems, their proof of this result is only based upon properties of $\Psi$ and therefore holds regardless of  the sample collected. Their proof relies on the fact that two random variables $X$ and $Y$ with compact support are equal in distribution if and only if $E[f(Y)]=E[f(X)]$ for every continuous function $f$, and uses techniques from the theory of Reproducing Kernel Hilbert Spaces \citep[see, e.g.,][for a general exposition]{Berlinet&ThomasAgnan2011}. We invite interested readers to consult \cite{Grettonetal2006} -- and, in particular, Theorem 3 therein -- for additional details. The definition of the MMD we utilize is based on the univariate Gaussian kernel with unit bandwidth, which is appropriate in view of \cite{Steinwart2002}. The results we present in this paper can be generalized to the MMD based on a Gaussian kernel of arbitrary bandwidth by simply rescaling the mappings $R$ and $S$.

\subsection{First-order differentiability}

To develop a test of $\mathscr{H}_0$, we will first construct an estimator $\psi_n$ of $\psi_0$. In order to avoid restrictive model assumptions, we wish to use flexible estimation techniques in estimating $P_0$ and therefore $\psi_0$. To control the operating characteristics of our test, it will be crucial to understand how to generate a parametric-rate estimator of $\psi_0$. For this purpose, it is informative to first investigate the pathwise differentiability of $\Psi$ as a parameter from $\mathscr{M}$ to $\mathbb{R}$.

So far, we have not specified restrictions on the mappings $P\mapsto R_P$ and $P\mapsto S_P$. However, in our developments, we will require these mappings to satisfy certain regularity conditions. Specifically, we will restrict our attention to elements of the class $\mathscr{S}$ of all mappings $T$ for which there exists some measurable function $X^T$ defined on $\mathscr{O}$, e.g. $X^T(o)=X^T(w,a,y)=w$, such that
\begin{enumerate}[label=(S\arabic*)]
	\item $T_P$ is a measurable mapping with domain $\{X^T(o) : o\in\mathscr{O} \}$ and range contained in $[-b,b]$ for some $0\leq b<\infty$ independent of $P$; \label{it:Sbdd}
	\item for all submodels $dP_t/dP= 1+th$ with bounded $h$ with $Ph=0$, there exists some $\delta>0$ and a set $\mathscr{O}_1\subseteq \mathscr{O}$ with $P_0(\mathscr{O}_1)=1$ such that, for all $(o,t_1)\in\mathscr{O}_1\times (-\delta,\delta)$, $t\mapsto T_{P_t}(x^T)$ is twice differentiable at $t_1$ with uniformly bounded (in $x^T$) first and second derivatives;\label{it:Sderivsbdd}
	\item for any $P\in\mathscr{M}$ and submodel $dP_t/dP= 1+th$ for bounded $h$  with $Ph=0$, there exists a function $D_P^T : \mathcal{O}\rightarrow\mathbb{R}$ uniformly bounded (in $P$ and $o$) such that $\int D_P^T(o)dP(o| x^T)=0$ for almost all $o\in\mathscr{O}$ and
\[\left.\frac{d}{dt}T_{P_t}(x^T)\right|_{t=0}= \int D_P^T(o) h(o) dP(o|x^T)\ .\]
\label{it:Spd}\end{enumerate} \condref{it:Sbdd} ensures that $T$ is bounded and only relies on a summary measure of an observation $O$. \condref{it:Sderivsbdd} ensures that we will be able to interchange differentiation and integration when needed. \condref{it:Spd} is a conditional (and weaker) version of pathwise differentiability in that the typical inner product representation only needs to hold for the conditional distribution of $O$ given $X^T$ under $P_0$. We will verify in \secref{sec:illustr} that these conditions hold in the context of the motivating examples presented earlier.

\begin{remark} \label{rem:restrictiveconds}
As a caution to the reader, we warn that simultaneously satisfying \ref{it:Sbdd} and \ref{it:Spd} may at times be restrictive. For example, if the observed data unit is $O\triangleq (W(1),W(2),Y)$, the parameter \[T_P(o)\triangleq E_P\left[Y\mid W(1)=w(1),W(2)=w(2)\right]-E_P\left[Y\mid W(1)=w(1)\right]\] cannot generally satisfy both conditions. In Section 5, we discuss this example further and provide a means to tackle this problem using the techniques we have developed. In concluding remarks, we discuss a weakening of our conditions, notably by replacing $\mathscr{S}$ by the linear span of elements in $\mathscr{S}$. Consideration of this larger class significantly complicates the form of the estimator we propose in \secref{sec:est}.\qed
\end{remark}

We are now in a position to discuss the pathwise differentiability of $\Psi$. For any elements $T,U\in\mathscr{S}$, we define 
\begin{align*}
\Gamma_P^{TU}(o_1,o_2)\triangleq& \bigg[2\left[T_P(o_1)-U_P(o_2)\right]\left[D_P^U(o_2) -D_P^T(o_1)\right] + 1 \\
&- \left\{4\left[T_P(o_1)-U_P(o_2)\right]^2-2\right\}D_P^T(o_1)D_P^U(o_2)\bigg]e^{-[T_P(o_1)-U_P(o_2)]^2}\ . 
\end{align*} and set $\Gamma_P\triangleq \Gamma_P^{RR} - \Gamma_P^{RS} - \Gamma_P^{SR} + \Gamma_P^{SS}$. Note that $\Gamma_P$ is symmetric for any $P\in\mathscr{M}$. For brevity, we will write $\Gamma_0^{TU}$ and $\Gamma_0$ to denote $\Gamma_{P_0}^{TU}$ and $\Gamma_{P_0}$, respectively. The following theorem characterizes the first-order behavior of $\Psi$ at an arbitrary $P\in\mathscr{M}$.
\begin{theorem}[First-order pathwise differentiability of $\Psi$ over $\mathscr{M}$] \label{thm:pd1}
If $R,S\in\mathscr{S}$, the parameter $\Psi:\mathscr{M}\rightarrow \mathbb{R}$ is pathwise differentiable at $P\in\mathscr{M}$ with first-order canonical gradient given by $D_1^{\Psi}(P)(o)\triangleq 2\left[\int \Gamma_P(o,o_2)dP(o_2)-\Psi(P)\right]$.
\end{theorem}
Under some conditions, it is straightforward to construct an asymptotically linear estimator of $\psi_0$ with influence function $D_1^\Psi(P_0)$, that is, an estimator $\psi_n$ of $\psi_0$ such that
\begin{align*}
\psi_n-\psi_0&= \frac{1}{n}\sum_{i=1}^n D_{1}^{\Psi}(P_0)(O_i) + o_{P_0}(n^{-1/2})\ .
\end{align*}
For example, the one-step Newton-Raphson bias correction procedure \citep[see, e.g.,][]{Pfanzagl1982} or targeted minimum loss-based estimation \citep[see, e.g.,][]{vanderLaan&Rose11} can be used for this purpose.  If the above representation holds and the variance of $D_1^{\Psi}(P_0)(O)$ is positive, then $\sqrt{n}\left(\psi_n-\psi_0\right)\rightsquigarrow N(0,\sigma_0^2)$, where the symbol $\rightsquigarrow$ denotes convergence in distribution and we write $\sigma_0^2\triangleq P_0\left[D_1^{\Psi}(P_0)^2\right]$.
If $\sigma_0$ is strictly positive and can be consistently estimated, Wald-type confidence intervals for $\psi_0$ with appropriate asymptotic coverage can be constructed.

The situation is more challenging if $\sigma_0=0$. In this case, $\sqrt{n}\left(\psi_n-\psi_0\right)\rightarrow 0$ in probability and typical Wald-type confidence intervals will not be appropriate. Because $D_1^{\Psi}(P_0)(O)$ has mean zero under $P_0$, this happens if and only if $D_1^{\Psi}(P_0)\equiv 0$. The following lemma provides necessary and sufficient conditions under which $\sigma_0=0$.
\begin{corollary}[First-order degeneracy under $\mathscr{H}_0$] \label{cor:D1degen}
If $R,S\in\mathscr{S}$, it will be the case that $\sigma_0=0$ if and only if either \hypertarget{it:condH0}{(i)} $\mathscr{H}_0$ holds, or \hypertarget{it:conddegen}{(ii)} $R_0(O)$ and $S_0(O)$ are degenerate, i.e. almost surely constant but not necessarily equal, with $D_0^R\equiv D_0^S$.
\end{corollary}
The above results rely in part on knowledge of $D_0^R$ and $D_0^S$. It is useful to note that, in some situations, the computation of $D_P^T(o)$ for a given $T\in\mathscr{S}$ and $P\in\mathscr{M}$ can be streamlined. This is the case, for example, if $P\mapsto T_P$ is invariant to fluctuations of the marginal distribution of $X^T$, as it seems \ref{it:Spd} may suggest. Consider obtaining iid samples of increasing size from the conditional distribution of $O$ given $X^T=x^T$ under $P$, so that all individuals have observed $X^T=x^T$. Consider the fluctuation submodel $dP_t(o|x^T)\triangleq \left[1 + th(o)\right]dP(o|x^T)$ for the conditional distribution, where $h$ is uniformly bounded and $\int h(o)dP(o|x^T)=0$. Suppose that (i) $P\mapsto T_P(x^T)$ is differentiable at $t=0$ with respect to the above submodel and (ii) this derivative satisfies the inner product representation
\begin{align*}
\left.\frac{d}{dt}T_{P_t}(x^T)\right|_{t=0}&= \int \tilde{D}_P^T(o|x^T)h(o)dP(o|x^T)
\end{align*} for some uniformly bounded function $o\mapsto \tilde{D}_P^T(o|x^T)$ with $\int \tilde{D}_P^T(o|x^T)dP(o|x^T)=0$. If the above holds for all $x^T$, we may take $D_P^T(o)=\tilde{D}_P^T(o|x^T)$ for all $o$ with $X^T(o)=x^T$. If $D_P^T$ is uniformly bounded in $P$, \ref{it:Spd} then holds. 

In summary, the above discussion suggests that, if $T$ is invariant to fluctuations of the marginal distribution of $X^T$, \ref{it:Spd} can be expected to hold if there exists a regular, asymptotically linear estimator of each $T_P(x^T)$ under iid sampling from the conditional distribution of $O$ given $X^T=x^T$ implied by $P$.
 


\begin{remark}
If $T$ is invariant to fluctuations of the marginal distribution of $X^T$, one can also expect \ref{it:Spd} to hold if $P\mapsto \int T_P(X^T(o))dP(o)$ is pathwise differentiable with canonical gradient uniformly bounded in $P$ and $o$ in the model in which the marginal distribution of $X$ is known. The canonical gradient in this model is equal to $D_P^T$.\qed
\end{remark}

\subsection{Second-order differentiability and asymptotic representation}\label{sec:2ndorder}

As indicated above, if $\sigma_0=0$, the behavior of $\Psi$ around $P_0$ cannot be adequately characterized by a first-order analysis. For this reason, we must investigate whether $\Psi$ is  second-order differentiable. As we discuss below, under $\mathscr{H}_0$, $\Psi$ is indeed second-order pathwise differentiable at $P_0$ and admits a useful second-order asymptotic representation.

\begin{theorem}[Second-order pathwise differentiability under $\mathscr{H}_0$] \label{thm:pd2}
If $R,S\in\mathscr{S}$ and $\mathscr{H}_0$ holds, the parameter $\Psi:\mathscr{M}\rightarrow\mathbb{R}$ is second-order pathwise differentiable at $P_0$ with second-order canonical gradient $D_2^{\Psi}(P_0)\triangleq 2\Gamma_0$.
\end{theorem}
It is easy to confirm that $\Gamma_0$, and thus $D_2^{\Psi}$, is one-degenerate under $\mathscr{H}_0$ in the sense that $\int \Gamma_0(o,o_2)dP_0(o_2)=\int \Gamma_0(o_1,o)dP_0(o_1)=0$ for all $o$. This is shown as follows. For any $T,U\in\mathscr{S}$, the law of total expectation conditional on $X^U$ and fact that $\int D_0^U(o)dP_0(o|x^U)=0$ yields that
\begin{align*}
\int &\Gamma_0^{TU}(o,o_2)dP_0(o_2) \\
&= \int \left\{1-2\left[T_0(o)-U_0(o_2)\right]D_0^T(o)\right\}e^{-\left[T_0(o)-U_0(o_2)\right]^2}dP_0(o_2),
\end{align*}
where we have written $\Gamma_0^{TU}$ to denote $\Gamma_{P_0}^{TU}$. Since $\int f(R_0(o))dP_0(o)=\int f(S_0(o))dP_0(o)$ for each measurable function $f$ when $S_0(O)\eqd T_0(O)$, this then implies that $\int \Gamma_0^{RS}(o,o_2)dP_0(o_2)=\int \Gamma_0^{RR}(o,o_2)dP(o_2)$ and $\int \Gamma_0^{SR}(o,o_2)dP_0(o_2)=\int \Gamma_0^{SS}(o,o_2)dP_0(o_2)$ under $\mathscr{H}_0$. Hence, it follows that $\int \Gamma_0(o,o_2)dP_0(o_2)=0$ under $\mathscr{H}_0$ for any $o$.

If second-order pathwise differentiability held in a sufficiently uniform sense over $\mathscr{M}$, we would expect \begin{align}
\Rem_P^\Psi\ \triangleq \ \Psi(P)-\Psi(P_0)-(P-P_0)D_1^\Psi(P)+\frac{1}{2}(P-P_0)^2D_2^\Psi(P) \label{rem}
\end{align} to be a third-order remainder term. However, second-order pathwise differentiability has only been established under the null, and in fact, it appears that $\Psi$ may not generally be second-order pathwise differentiable under the alternative. As such, $D_2^\Psi$ may not even be defined under the alternative. In writing \eqref{rem}, we either naively set $D_2^\Psi(P)\triangleq 2\Gamma_P$, which is not appropriately centered to be a candidate second-order gradient, or instead take  $D_2^\Psi$ to be the centered extension
\begin{align*}
(o_1,o_2)\mapsto 2\left[\Gamma_P(o_1,o_2)-\int \Gamma_P(o_1,o)dP(o)-\int \Gamma_P(o,o_2)dP(o)+P^2\Gamma_P\right].
\end{align*}
Both of these choices yield the same expression above because the product measure $(P-P_0)^2$ is self-centering. The need for an extension renders it a priori unclear whether as $P$ tends to $P_0$ the behavior of $\Rem_P^\Psi$ is similar to what is expected under more global second-order pathwise differentiability. Using the fact that $\Psi(P)=P^2\Gamma_P$, we can simplify the expression in \eqref{rem} to \begin{equation}\Rem_P^\Psi\ =\ P_0^2\Gamma_P-\psi_0\ . \label{simpleremainder}
\end{equation} As we discuss below, this remainder term can be bounded in a useful manner, which allows us to determine that it is indeed third-order. 

For all $T\in\mathscr{S}$, $P\in\mathscr{M}$ and $o\in\mathscr{O}$, we define
\begin{align*}
\Rem_P^T(o)\ \triangleq \ T_P(o)-T_0(o) + \int D_P^T(o_1)\left[dP(o_1|x^T)-dP_0(o_1|x^T)\right]
\end{align*} as the remainder from the linearization of $T$ based on the conditional gradient $D_P^T$.
Typically, $\Rem_P^T(o)$ is a second-order term. Further consideration of this term in the context of our motivating examples is described in \secref{sec:illustr}.
 Furthermore, we define
\begin{align*}
L_P^{RS}(o)\ &\triangleq \ \max\left\{|\Rem_P^R(o)|,|\Rem_P^S(o)|\right\} \\
M_P^{RS}(o)\ &\triangleq \  \max\left\{|R_P(o)-R_0(o)|,|S_P(o)-S_0(o)|\right\} .
\end{align*} For any given function $f : \mathcal{O}\rightarrow\mathbb{R}$, we denote by $\norm{f}_{p,P_0}\triangleq \left[\int |f(o)|^pdP_0(o)\right]^{1/p}$ the $L^p(P_0)$-norm and use the symbol $\lesssim$ to denote `less than or equal to up to a positive multiplicative constant'. The following theorem provides an upper bound for the remainder term of interest.

\begin{theorem}[Upper bounds on remainder term] \label{thm:remthirdorder}
For each $P\in\mathscr{M}$, the remainder term admits the following upper bounds: \begin{align*}
\mbox{Under }\mathscr{H}_0:\ \ &|\Rem_P^{\Psi}|\ \lesssim\ K_{0P}\triangleq \twonorm{L_P^{RS}} \twonorm{M_P^{RS}} + \onenorm{L_P^{RS}}^2 + \fournorm{M_P^{RS}}^4\\
\mbox{Under }\mathscr{H}_1:\ \ &|\Rem_P^{\Psi}|\ \lesssim\ K_{1P}\triangleq \onenorm{L_P^{RS}} + \twonorm{M_P^{RS}}^2\ .
\end{align*}

\end{theorem}
To develop a test procedure, we will require an estimator of $P_0$, which will play the role of $P$ in the above expressions. It is helpful to think of parametric model theory when interpreting the above result, with the understanding that certain smoothing methods, such as higher-order kernel smoothing, can achieve near-parametric rates in certain settings. In a parametric model, we could often expect $\norm{L_P^{RS}}_{p,P_0}$ and $\norm{M_P^{RS}}_{p,P_0}$ to be $O_{P_0}(n^{-1})$ and $O_{P_0}(n^{-1/2})$, respectively, for $p\ge 1$. Thus, the above theorem suggests that the approximation error may be $O_P(n^{-3/2})$ in a parametric model under $\mathscr{H}_0$. In some examples, it is reasonable to expect that $L_P^{RS}\equiv 0$ for a large class of distributions $P$. In such cases, the upper bound on $\Rem_P^{\Psi}$ simplifies to $\fournorm{M_P^{RS}}^4$ under $\mathscr{H}_0$, which under a parametric model is often $O_{P_0}(n^{-2})$.

\section{Proposed test: formulation and inference under the null} \label{sec:est}

\subsection{Formulation of test}

We begin by constructing an estimator of $\psi_0$ from which a test can then be devised. Using the fact that $\Psi(P)=P^2\Gamma_P$, as implied by \eqref{simpleremainder}, we note that if $\Gamma_0$ were known, the U-statistic $\mathbb{U}_n\Gamma_0$ would be a natural estimator of $\psi_0$, where $\mathbb{U}_n$ denotes the empirical measure that places equal probability mass on each of the $n(n-1)$ points $(O_i,O_j)$ with $i\neq j$. In practice, $\Gamma_0$ is unknown and must be estimated.  This leads to the estimator $\psi_n\triangleq \mathbb{U}_n\Gamma_n$, where we write $\Gamma_n\triangleq \Gamma_{\hat{P}_n}$ for some estimator $\hat{P}_n$ of $P_0$ based on the available data. Since a large value of $\psi_n$ is inconsistent with $\mathscr{H}_0$, we will reject $\mathscr{H}_0$ if and only if $\psi_n>c_n$ for some appropriately chosen cutoff $c_n$.

In the nonparametric model considered, it may be necessary, or at the very least desirable, to utilize a data-adaptive estimator $\hat{P}_n$ of $P_0$ when constructing $\Gamma_n$. Studying the large-sample properties of $\psi_n$ may then seem particularly daunting since at first glance we may be led to believe that the behavior of $\psi_n-\psi_0$ is dominated by $P_0^2\left(\Gamma_n-\Gamma_0\right)$.  However, this is not the case. As we will see, under some conditions,  $\psi_n-\psi_0$ will approximately behave like $\left(\mathbb{U}_n-P_0^2\right)\Gamma_0$. Thus, there will be no contribution of $\hat{P}_n$ to the asymptotic behavior of $\psi_n-\psi_0$. Though this result may seem counterintuitive, it arises because $\Psi(P)$ can be expressed as $P^2\Gamma_P$ with $\Gamma_P$ a second-order gradient (or rather an extension thereof) up to a proportionality constant. More concretely, this surprising finding is a direct consequence of \eqref{simpleremainder}.

As further support that $\psi_n$ is a natural test statistic, even when a data-adaptive estimator $\hat{P}_n$ of $P_0$ has been used, we note that $\psi_n$ could also have been derived using a second-order one-step Newton-Raphson construction, as described in \cite{RobinsetTchetgenvanderVaart08}. The latter is given by \[\psi_{n,NR}\triangleq \Psi(\hat{P}_n)+P_nD^\Psi_1(\hat{P}_n)+\frac{1}{2}\mathbb{U}_nD^\Psi_2(\hat{P}_n)\ ,\] where we use the centered extension of $D_2^\Psi$ as discussed in \secref{sec:2ndorder}. Here and throughout, $P_n$ denotes the empirical distribution. It is straightforward to verify that indeed $\psi_{n}=\psi_{n,NR}$.

\subsection{Inference under the null} \label{sec:typeIctrl}

\subsubsection{Asymptotic behavior}

For each $P\in\mathscr{M}$, we let $\tilde{\Gamma}_P$ be the $P_0$-centered modification of $\Gamma_P$ given by
\begin{align*}
\tilde{\Gamma}_P(o_1,o_2)\triangleq \Gamma_P(o_1,o_2) - \int \Gamma_P(o_1,o)dP_0(o)  - \int \Gamma_P(o,o_2)dP_0(o) +  P_0^2 \Gamma_P
\end{align*} and denote $\tilde{\Gamma}_{P_0}$ by $\tilde{\Gamma}_0$. While $\tilde{\Gamma}_0=\Gamma_0$ under $\mathscr{H}_0$, this is not true more generally. Below, we use $\Rem_n^{\Psi}$ and $\tilde{\Gamma}_n$ to respectively denote $\Rem_P^{\Psi}$ and $\tilde{\Gamma}_P$ evaluated at $P=\hat{P}_n$. Straightforward algebraic manipulations allows us to write 
\begin{align}
\psi_n-\psi_0\ &=\ \mathbb{U}_n\Gamma_n-\psi_0\ =\ \mathbb{U}_n\Gamma_n-P_0^2\Gamma_n+P_0^2\Gamma_n-\psi_0 \nonumber\\
&=\ \left(\Un-P_0^2\right)\Gamma_n + \Rem_n^{\Psi} \nonumber \\
&=\  \Un \Gamma_0 + 2\left(P_n-P_0\right)P_0 \Gamma_n + \Un\left(\tilde{\Gamma}_n-\Gamma_0\right) + \Rem_n^{\Psi}\ . \label{eq:nullexpansion}
\end{align}
Our objective is to show that $n\left(\psi_n-\psi_0\right)$ behaves like $n\Un \Gamma_0$ as $n$ gets large under $\mathscr{H}_0$. In view of \eqref{eq:nullexpansion}, this will be true, for example, under conditions ensuring that
\begin{enumerate}[label=C\arabic*)]
	\item $n(P_n-P_0)P_0 \Gamma_n=o_{P_0}(1)$\ \ (empirical process and consistency conditions); \label{it:1ordempproc}
	\item $n\Un \left(\tilde{\Gamma}_n-\Gamma_0\right)=o_{P_0}(1)$\ \ ($U$-process and consistency conditions); \label{it:2ordempproc}
	\item $n\Rem_n^{\Psi}=o_{P_0}(1)$\ \ (consistency and rate conditions). \label{it:remnegl}
\end{enumerate}
We have already argued that \ref{it:remnegl} is reasonable in many examples of interest, including those presented in this paper. \cite{Nolan&Pollard1987,Nolan&Pollard1988} developed a formal theory that controls terms of the type appearing in \ref{it:2ordempproc}. In Appendix B.1 we restate specific results from these authors which are useful to study \ref{it:2ordempproc}.  Finally, the following lemma gives sufficient conditions under which \ref{it:1ordempproc} holds. We first set $K_{1n}\triangleq  \onenorm{L_{\hat{P}_n}^{RS}} + \twonorm{M_{\hat{P}_n}^{RS}}^2$.

\begin{lemma}[Sufficient conditions for \ref{it:1ordempproc}] \label{lem:firstordnegl}
Suppose that  $o_1\mapsto \int \Gamma_n(o_1,o)dP_0(o)/K_{1n}$, defined to be zero if $K_{1n}=0$, belongs to a $P_0$-Donsker class \citep{vanderVaartWellner1996} with probability tending to $1$. Then, under $\mathscr{H}_0$,
\[
(P_n-P_0)P_0 \Gamma_n= O_{P_0}\left(\frac{K_{1n}}{\sqrt{n}}\right)
\] and thus \ref{it:1ordempproc} holds whenever $K_{1n}=o_{P_0}(n^{-1/2})$.
\end{lemma}

The following theorem describes the asymptotic distribution of $n\psi_n$ under the null hypothesis whenever conditions \ref{it:1ordempproc}, \ref{it:2ordempproc} and \ref{it:remnegl} are satisfied.

\begin{theorem}[Asymptotic distribution under $\mathscr{H}_0$] \label{thm:asympt}
Suppose that \ref{it:1ordempproc}, \ref{it:2ordempproc} and \ref{it:remnegl} hold. Then, under $\mathscr{H}_0$,
\[
n\psi_n= n \Un \Gamma_0 + o_{P_0}(1)\rightsquigarrow \sum_{k=1}^\infty \lambda_k \left(Z_k^2-1\right) ,
\]
where $\{\lambda_k\}_{k=1}^{\infty}$ are the eigenvalues of the integral operator $h(o)\mapsto \int \Gamma_0(o_1,o)h(o)dP_0(o_1)$ repeated according to their multiplicity, and $\{Z_k\}_{k=1}^\infty$ is a sequence of independent standard normal random variables. Furthermore, all of these eigenvalues are nonnegative under $\mathscr{H}_0$.
\end{theorem}

We note that by employing a sample splitting procedure -- namely, estimating $\Gamma_0$ on one portion of the sample and constructing the $U$-statistic based on the remainder of the sample -- it is possible to eliminate the $U$-process conditions required for  \ref{it:2ordempproc}. In such a case, satisfaction of \ref{it:2ordempproc} only requires convergence of $\tilde{\Gamma}_n$ to $\Gamma_0$ with respect to the $L^2(P_0^2)$-norm.

\begin{remark}In \hyperlink{ex:3}{Example 3}, sample splitting may prove particularly important when the estimator of $E_{P_0}\left(Y\mid W=w\right)$ is chosen as the minimizer of an empirical risk since in finite samples the  bias induced by using the same residuals $y-E_{\hat{P}_n}\left(Y\mid W=w\right)$ as those in the definition of $D_{\hat{P}_n}^R(o)$ may be significant. Thus, without some form of sample splitting, the finite sample performance of $\psi_n$ may be poor even under the conditions stated in Appendix B.1.\qed
\end{remark}

\subsubsection{Estimation of the test cutoff} \label{sec:estcutoff}

As indicated above, our test consists of rejecting $\mathscr{H}_0$ if and only if $\psi_n$ is larger than some cutoff $c_n$. We wish to select $c_n$ to yield a non-conservative test at level $\alpha\in(0,1)$. In view of \autoref{thm:asympt}, denoting by $q_{1-\alpha}$ the $1-\alpha$ quantile of the described limit distribution, the cutoff $c_n$ should be chosen to be $q_{1-\alpha}/n$. We thus reject $\mathscr{H}_0$ if and only if $n\psi_n>q_{1-\alpha}$. As described in the following corollary, $q_{1-\alpha}$ admits a very simple form when $S_P\equiv 0$ for all $P$.
\begin{corollary}[Asymptotic distribution under $\mathscr{H}_0$, $S$ degenerate] \label{cor:Seq0}
Suppose that \ref{it:1ordempproc}, \ref{it:2ordempproc} and \ref{it:remnegl} hold, that $S_P\equiv 0$ for all $P\in\mathscr{M}$, and that $\sigma_R^2\triangleq Var_{P_0}\left[D_0^R(O)\right]>0$. Then, under $\mathscr{H}_0$,
\[
\frac{n\psi_n}{2\sigma_R^2}\rightsquigarrow Z^2-1,
\]
where $Z$ is a standard normal random variable. It follows then that $q_{1-\alpha}=2\sigma_R^2 (z^2_{1-\alpha/2}-1)$, where $z_{1-\alpha/2}$ is the $(1-\alpha/2)$ quantile of the standard normal distribution.
\end{corollary}

The above corollary gives an expression for $q_{1-\alpha}$ that can easily be consistently estimated from the data. In particular, one can use $\hat{q}_{1-\alpha}\triangleq 2(z_{1-\alpha/2}^2-1) P_n D^R(\hat{P}_n)^2$ as an estimator of $q_{1-\alpha}$, whose consistency can be established under a Glivenko-Cantelli and consistency condition on the estimator of $D_0^R$. However, in general, such a simple expression will not exist. \cite{Grettonetal2009} proposed estimating the eigenvalues $\nu_k$ of the centered Gram matrix and then computing $\hat{\lambda}_k\triangleq \nu_k/n$. In our context, the eigenvalues $\nu_k$ are those of the $n\times n$ matrix $G\triangleq \{G_{ij}\}_{1\leq i,j\leq n}$ with entries defined as
\begin{equation}
G_{ij}\ \triangleq \ \Gamma_n(O_i,O_j) - \frac{1}{n}\sum_{k=1}^{n}\Gamma_n(O_k,O_j) - \frac{1}{n}\sum_{\ell=1}^{n}\Gamma_n(O_i,O_\ell) + \frac{1}{n^2}\sum_{k=1}^{n}\sum_{\ell=1}^{n}\Gamma_n(O_k,O_\ell)\ . \label{eq:gram}
\end{equation}
Given these $n$ eigenvalue estimates $\hat{\lambda}_{1},...,\hat{\lambda}_n$, one could then simulate from $\sum_{k=1}^n \hat{\lambda}_k (Z_k^2-1)$ to approximate $\sum_{k=1}^\infty \lambda_k (Z_k^2-1)$. While this seems to be a plausible approach, a formal study establishing regularity conditions under which this procedure is valid is beyond the scope of this paper. We note that it also does not fall within the scope of results in \cite{Grettonetal2009} since their kernel does not depend on estimated nuisance parameters. We refer the reader to \cite{Franz2006} for possible sufficient conditions under which this approach may be valid.

In practice, it suffices to give a data-dependent asymptotic upper bound on $q_{1-\alpha}$. We will refer to $\hat{q}_{1-\alpha}^{ub}$, which depends on $P_n$, as an asymptotic upper bound of $q_{1-\alpha}$ if
\begin{align}
\limsup_{n\rightarrow\infty} P_0^n\left(n\psi_n>\hat{q}_{1-\alpha}^{ub}\right)\ \le\ 1-\alpha\ .\label{eq:cutoffub}
\end{align}
If $q_{1-\alpha}$ is consistently estimated, one possible choice of $\hat{q}_{1-\alpha}^{ub}$ is this estimate of $q_{1-\alpha}$ -- the inequality above would also become an equality provided the conclusion of \autoref{thm:asympt} holds. It is easy to derive a data-dependent upper bound with this property using Chebyshev's inequality. To do so, we first note that
\begin{align*}
Var_{P_0}\left[\sum_{k=1}^\infty \lambda_k \left(Z_k^2-1\right)\right]\ =\ \sum_{k=1}^\infty \lambda_k^2 Var_{P_0}\left(Z_k^2\right)\ =\ 2\sum_{k=1}^\infty \lambda_k^2\ =\ 2P_0^2\Gamma_0^2\ ,
\end{align*}
where we have interchanged the variance operation and the limit using the $L^2$ martingale convergence theorem and the last equality holds because $\lambda_k$, $k=1,2,\ldots$, are the eigenvalues of the Hilbert-Schmidt integral operator with kernel $\tilde{\Gamma}_0$. Under mild regularity conditions, $P_0^2\Gamma_0^2$ can be consistently estimated using $\Un \Gamma_n^2$. Provided $P_0^2\Gamma_0^2>0$, we find that
\begin{align}
\left(2\mathbb{U}_n\Gamma_n^2\right)^{-1/2}n\psi_n\rightsquigarrow \left(2P_0^2 \Gamma_0^2\right)^{-1/2}\sum_{k=1}^\infty \lambda_k \left(Z_k^2-1\right)\ , \label{eq:var1}
\end{align}
where the limit variate has mean zero and unit variance. The following theorem gives a valid choice of $\hat{q}_{0.95}^{ub}$.
\begin{theorem}
Suppose that \ref{it:1ordempproc}, \ref{it:2ordempproc} and \ref{it:remnegl} hold. Then, under $\mathscr{H}_0$ and provided $\Un \Gamma_n^2\rightarrow P_0^2 \Gamma_0^2>0$ in probability, $\hat{q}_{0.95}^{ub}\triangleq 6.2\cdot \left(\Un \Gamma_n^2\right)^{1/2}>q_{0.95}$ is a valid upper bound in the sense of \eqref{eq:cutoffub}.\end{theorem}
The proof of the result follows immediately by noting that $P(X>t)\leq (1+t^2)^{-1}$ for any random variable $X$ with mean zero and unit variance in view of the one-sided Chebyshev's inequality. This  illustrates concretely that we can obtain a consistent test that controls type I error. In practice, we recommend either using the result of \autoref{cor:Seq0} whenever possible or estimating the eigenvalues of the matrix in (\ref{eq:gram}).

We note that the condition $\sigma_R^2>0$ holds in many but not all examples of interest. Fortunately, the plausibility of this assumption can be evaluated analytically. In \secref{sec:illustr}, we show that this condition does not hold in \hyperlink{ex:4}{Example 4} and provide a way forward despite this.

\section{Asymptotic behavior under the alternative} \label{sec:local}

\subsection{Consistency under a fixed alternative} \label{sec:testconsistent}
We present two analyses of the asymptotic behavior of our test under a fixed alternative. The first relies on $\hat{P}_n$ providing a good estimate of $P_0$. Under this condition, we give an interpretable limit distribution that provides insight into the behavior of our estimator under the alternative. As we show, surprisingly, $\hat{P}_n$ need not be close to $P_0$ to obtain an asymptotically consistent test, even if the resulting estimate of $\psi_0$ is nowhere near the truth. In the second analysis, we give more general conditions under which our test will be consistent if $\mathscr{H}_1$ holds.

\subsubsection{Nuisance functions have been estimated well}
As we now establish, our test has power against all alternatives $P_0$ except for the fringe cases discussed in \autoref{cor:D1degen} with $\Gamma_0$ one-degenerate. We first note that
\[\psi_n-\psi_0\ =\ \Un \Gamma_n-\psi_0\ =\ 2(P_n-P_0) P_0 \Gamma_n+ \Un \tilde{\Gamma}_n+ \Rem_P^\Psi \ .\]
When scaled by $\sqrt{n}$, the leading term on the right-hand side follows a mean zero normal distribution under regularity conditions.  The second summand is typically $O_{P_0}(n^{-1})$ under certain conditions, for example, on the entropy of the class of plausible realizations of the random function $(o_1,o_2)\mapsto \Gamma_n(o_1,o_2)$ \citep{Nolan&Pollard1987,Nolan&Pollard1988}.  In view of the second statement in \autoref{thm:remthirdorder}, the third summand is a second-order term that will often be negligible, even after scaling by $\sqrt{n}$. As such, under certain regularity conditions, the leading term in the representation above determines the asymptotic behavior of $\psi_n$, as described in the following theorem.
\begin{theorem}[Asymptotic distribution under $\mathscr{H}_1$] \label{thm:altasdist}
Suppose that $K_{1n}=o_{P_0}(n^{-1/2})$, that $\Un \tilde{\Gamma}_n=o_{P_0}(n^{-1/2})$, and furthermore, that $o\mapsto \int \Gamma_n(o_1,o)dP_0(o)$ belongs to a fixed $P_0$-Donsker class with probability tending to $1$ while $\twonorm{P_0\left(\Gamma_n-\Gamma_0\right)}=o_{P_0}(1)$. If $\mathscr{H}_1$ holds, we have that $
\sqrt{n}\left(\psi_n-\psi_0\right)\rightsquigarrow N\left(0,\tau^2\right)$, where $\tau^2\triangleq 4Var_{P_0}\left[\int \Gamma_0(O,o)dP_0(o)\right]$.
\end{theorem}
In view of the results of \secref{sec:rep}, $\tau^2$ coincides with $\sigma_0^2$, the efficiency bound for regular, asymptotically linear estimators in a nonparametric model. Hence, $\psi_n$ is an asymptotically efficient estimator of $\psi_0$ under $\mathscr{H}_1$. Sufficient conditions for $\int \Gamma_n(o_1,o)dP_0(o)$ to belong to a fixed $P_0$-Donsker class with probability approaching one are given in Appendix B.2.

The following corollary is trivial in light of \autoref{thm:altasdist}. It establishes that the test $n\psi_n>\hat{q}_{1-\alpha}^{ub}$ is consistent against (essentially) all alternatives provided the needed components of the likelihood are estimated sufficiently well.
\begin{corollary}[Consistency under a fixed alternative] \label{cor:testconsistent}
Suppose the conditions of \autoref{thm:altasdist}. Furthermore, suppose that $\tau^2>0$ and $\hat{q}_{1-\alpha}^{ub}=o_{P_0}(n)$. Then, under $\mathscr{H}_1$, the test $n\psi_n>\hat{q}_{1-\alpha}^{ub}$ is consistent in the sense that
\[
\lim_{n\rightarrow\infty}P_0^n\left(n\psi_n>\hat{q}_{1-\alpha}^{ub}\right)=1\ .
\]
\end{corollary}
The requirement that $\hat{q}_{1-\alpha}^{ub}=o_{P_0}(n)$ is very mild given that $q_{1-\alpha}$ will be finite whenever $R,S\in\mathscr{S}$. As such, we would not expect $\hat{q}_{1-\alpha}^{ub}$ to get arbitrarily large as sample size grows, at least beyond the extent allowed by our corollary. This suggests that most non-trivial upper bounds satisfying (\ref{eq:cutoffub}) will yield a consistent test.


\subsubsection{Nuisance functions have not been estimated well}
We now consider the case where the nuisance functions are not estimated well, in the sense that the consistency conditions of \autoref{thm:altasdist} do not hold. In particular, we argue that failure of these conditions does not necessarily undermine the consistency of our test. Let $\hat{q}_{1-\alpha}^{ub}$ be the estimated cutoff for our test, and suppose that $\hat{q}_{1-\alpha}^{ub}=o_{P_0}(n)$. Suppose also that $P_0^2\Gamma_n$ is asymptotically bounded away from zero in the sense that, for some $\delta>0$, $P_0^n\left(P_0^2\Gamma_n>\delta\right)$ tends to one.
This condition is reasonable given that $P_0^2 \Gamma_0>0$ if $\mathscr{H}_1$ holds and $\hat{P}_n$ is nevertheless a (possibly inconsistent) estimator of $P_0$. Assuming that $[\Un-P_0^2]\Gamma_n=O_{P_0}(n^{-1/2})$, which is true under entropy conditions on $\Gamma_n$ \citep{Nolan&Pollard1987,Nolan&Pollard1988}, we have that
\[P_0^n\left(n\psi_n > \hat{q}_{1-\alpha}^{ub}\right)\ =\ P_0^n\left(\sqrt{n}[\Un-P_0^2]\Gamma_n>\frac{\hat{q}_{1-\alpha}^{ub}}{\sqrt{n}}-\sqrt{n}P_0^2 \Gamma_n\right)\longrightarrow 1\ .
\]We have accounted for the random $n^{-1/2}\hat{q}_{1-\alpha}^{ub}$ term as in the proof of \autoref{cor:testconsistent}. Of course, this result is less satisfying than \autoref{thm:altasdist}, which provides a concrete limit distribution.

\subsection{Consistency under a local alternative} \label{sec:loccons}

We consider local alternatives of the form
\begin{align*}
dQ_n(o)&=\left[1 + n^{-1/2}h_n(o)\right] dP_0(o),
\end{align*}
 where $h_n\rightarrow h$ in $L_0^2(P_0)$ for some non-degenerate $h$ and $P_0$ satisfies the null hypothesis $\mathscr{H}_0$. Suppose that the conditions of \autoref{thm:asympt} hold. By Theorem 2.1 of \cite{Gregory1977}, we have that
\begin{align*}
n\Un \Gamma_0\ \overset{Q_n}{\rightsquigarrow}\ \sum_{k=1}^\infty \lambda_k\left[\left(Z_k+\langle f_k ,h \rangle\right)^2-1\right],
\end{align*}
where $\Un$ is the $U$-statistic empirical measure from a sample of size $n$ drawn from $Q_n$, $\langle\cdot,\cdot\rangle$ is the inner product in $L^2(P_0)$, $Z_k$ and $\lambda_k$ are as in \autoref{thm:asympt}, and $f_k$ is the eigenfunction corresponding to eigenvalue $\lambda_k$ described in \autoref{thm:asympt}. By the contiguity of $Q_n$, the conditions of \autoref{thm:asympt} yield that the result above also holds with $\Un \Gamma_0$ replaced by $\Un \Gamma_n$, our estimator applied to a sample of size $n$ drawn from $Q_n$.

If each $\lambda_k$ is non-negative, the limiting distribution under $Q_n$ stochastically dominates the asymptotic distribution under $P_0$, and furthermore, if $\langle f_k ,h \rangle\not=0$ for some $k$ with $\lambda_k>0$,  this dominance is strict. It is straightforward to show that, under the conditions of \autoref{thm:asympt}, the above holds if and only if $\liminf_{n}\sqrt{n}\Psi(Q_n)>0$, that is, if the sequence of alternatives is not too hard. Suppose that $\hat{q}_{1-\alpha}$ is a consistent estimate of $q_{1-\alpha}$. By Le Cam's third lemma, $\hat{q}_{1-\alpha}$ is consistent for $q_{1-\alpha}$ even when the estimator is computed on samples of size $n$ drawn from $Q_n$ rather than $P_0$. This proves the following theorem.
\begin{theorem}[Consistency under a local alternative]
Suppose that the conditions of \autoref{thm:asympt} hold. Then, under $\mathscr{H}_0$ and provided $\liminf_{n\rightarrow\infty}\sqrt{n}\Psi(Q_n)>0$, the proposed test is locally consistent in the sense that $\lim_{n\rightarrow\infty}Q_n\left(n\psi_n>\hat{q}_{1-\alpha}\right)>\alpha$, where $\hat{q}_{1-\alpha}$ is a consistent estimator of $q_{1-\alpha}$.
\end{theorem}

\section{Illustrations} \label{sec:illustr}

We now return to each of our examples. We first show that Examples 1, 2 and 3 satisfy the regularity conditions described in \secref{sec:rep}. Specifically, we show that all involved parameters $R$ and $S$ belong to $\mathscr{S}$ under reasonable conditions. Furthermore, we determine explicit remainder terms for the asymptotic representation used in each example and describe conditions under which these remainder terms are negligible. For any $T\in\mathscr{S}$, we will use the shorthand notation $\dot{T}_{\tilde{t}}(x^T)\triangleq  \left.\frac{d}{d t}T_{P_t}(x^T)\right|_{t=\tilde{t}}$ for $\tilde{t}$ in a neighborhood of zero.\\

\noindent \textbf{\hyperlink{ex:1}{Example 1}} (Continued).

The parameter $S$ with $S_P\equiv 0$ belongs to $\mathscr{S}$ trivially, with $D_P^S\equiv 0$. \condref{it:Sbdd} holds with $x^R(o)=w$. \condref{it:Sderivsbdd} holds using that $R_t(w)$ equals
\begin{align}
\sum_{a=0}^1 (-1)^{a+1}\int y\left\{\frac{1 + th_1(w,a,y) + t^2h_2(w,a,y)}{1 + tE_{P_0}[h_1(w,A,Y)] + t^2E_{P_0}[h_2(w,A,Y)]}\right\}dP_0(y| a,w)\ . \label{eq:Rtdiffblip}
\end{align}
Since we must only consider $h_1$ and $h_2$ uniformly bounded, for $t$ sufficiently small, we see that $R_t(w)$ is twice continuously differentiable with uniformly bounded derivatives. \condref{it:Spd} is satisfied by
\[
D_P^R(o)\triangleq  \frac{2a-1}{P\left(A=a\mid W=w\right)}\left\{y-E_{P}\left[Y\mid A=a,W=w\right]\right\}
\]
and $D_P^S\equiv 0$. If $\min_a P\left(A=a\mid W\right)$ is bounded away from zero with probability $1$ uniformly in $P$, it follows that $(P,o)\mapsto D_P^R(o)$ is uniformly bounded.

Clearly, we have that $\Rem_P^S\equiv 0$. We can also verify that $\Rem_P^R(o)$ equals
\begin{align*}
\sum_{\tilde{a}=0}^1 (-1)^{\tilde{a}}E_{P_0}&\Big\{\left[1-\frac{P_0\left(A=\tilde{a}\mid W\right)}{P\left(A=\tilde{a}\mid W\right)}\right] \\
&\;\;\;\;\times\left[E_P\left(Y\mid A,W\right)-E_{P_0}\left(Y\mid A,W\right)\right] \Big| A=\tilde{a},W=w\Big\}.
\end{align*}
The above remainder is double robust in the sense that it is zero if either the treatment mechanism (i.e., the probability of $A$ given $W$) or the outcome regression (i.e., the expected value of $Y$ given $A$ and $W$) is correctly specified under $P$. In a randomized trial where the treatment mechanism is known and specified correctly in $P$, we have that $\Rem_P^R\equiv 0$ and thus $L_P^{RS}\equiv 0$. More generally, an upper bound for $\Rem_P^R$ can be found using the Cauchy-Schwarz inequality to relate the rate of $\twonorm{\Rem_P^R}$ to the product of the $L^2(P_0)$-norm for the difference between each of the treatment mechanism and the outcome regression under $P$ and $P_0$.\\

\noindent \textbf{\hyperlink{ex:2}{Example 2}} (Continued).

For \ref{it:Sbdd} we take $x^R=x^S=w$. \condref{it:Sderivsbdd} can be verified using an expression similar to that in (\ref{eq:Rtdiffblip}). \condref{it:Spd} is satisfied by
\begin{align*}
D_P^R(o)\ &\triangleq \ \frac{a}{P\left(A=a\mid W=w\right)}\left[y-E_{P}\left(Y\mid A=a,W=w\right)\right] \\
D_P^S(o)\ &\triangleq \ \frac{1-a}{P\left(A=a\mid W=w\right)}\left[y-E_{P}\left(Y\mid A=a,W=w\right)\right]\ .
\end{align*}
If $\min_a P\left(A=a\mid W\right)$ is bounded away from zero with probability $1$ uniformly in $P$, both $(P,o)\mapsto D_P^R(o)$ and $(P,o)\mapsto D_P^S(o)$ are uniformly bounded.

Similarly to \hyperlink{ex:1}{Example 1}, we have that $\Rem_P^R(o)$ is equal to
\begin{align*}
E_{P_0}\left\{\left[1-\frac{P_0\left(A=1\mid W\right)}{P\left(A=1\mid W\right)}\right]\left[E_P\left(Y\mid A,W\right)-E_{P_0}\left(Y\mid A,W\right)\right]\ \middle|\ A=1,W=w\right\}.
\end{align*}
The remainder $\Rem_P^S(o)$ is equal to the above display but with $A=1$ replaced by $A=0$. The discussion about the double robust remainder term from \hyperlink{ex:1}{Example 1} applies to these remainders as well.\\

\noindent \textbf{\hyperlink{ex:3}{Example 3}} (Continued).

The parameter $S$ is the same as in \hyperlink{ex:1}{Example 1}. The parameter $R$ satisfies \ref{it:Sbdd} with $x^R(o)=w$ and \ref{it:Sderivsbdd} by an identity analogous to that used in \hyperlink{ex:1}{Example 1}. \condref{it:Spd} is satisfied by $D_P^R(o)\triangleq  y-E_P\left(Y\mid W=w\right)$. By the bounds on $Y$, $(P,o)\mapsto D_P^R(o)$ is uniformly bounded. Here, the remainder terms are both exactly zero: $\Rem_P^R\equiv \Rem_P^S\equiv 0$. Thus, we have that $L_P^{RS}\equiv 0$ in this example.\\

The requirement that $Var_{P_0}\left[D_0^R(O)\right]>0$ in \autoref{cor:Seq0}, and more generally that there exist a nonzero eigenvalue $\lambda_j$ for the limit distribution in \autoref{thm:asympt} to be non-degenerate, may at times present an obstacle to our goal of obtaining asymptotic control of the type I error. This is the case for \hyperlink{ex:4}{Example 4}, which we now discuss further. Nevertheless, we show that with a little finesse the type I error can still be controlled at the desired level for the given test. In fact, the test we discuss has type I error converging to zero, suggesting it may be noticeably conservative in small to moderate samples.  \\

\noindent \textbf{\hyperlink{ex:4}{Example 4}} (Continued).

In this example, one can take $x^R=w$ and $x^S=w(-k)$. Furthermore, it is easy to show that
\begin{align*}
D_P^R(o)&= Y-E_P[Y|W=w] \\
D_P^S(o)&= Y-E_P[Y|W(-k)=w(-k)]\ .
\end{align*}
 The first-order approximations for $R$ and $S$ are exact in this example as the remainder terms $\Rem_P^R$ and $\Rem_P^S$ are both zero. However, we note that if $E_P\left(Y\mid W\right)= E_P\left(Y\mid W(-k)\right)$ almost surely, it follows that $D_P^R\equiv D_P^S$. This implies that $\Gamma_0\equiv 0$ almost surely under $\mathscr{H}_0$. As such, under the conditions of \autoref{thm:asympt},  all of the eigenvalues in the limit distribution of $n\psi_n$ in \autoref{thm:asympt} are zero and $n\psi_n\rightarrow 0$ in probability. We are then no longer able to control the type I error at level $\alpha$, rendering our proposed test invalid.

Nevertheless, there is a simple albeit unconventional way to repair this example. Let $A$ be a Bernoulli random variable, independent of all other variables, with fixed probability of success $p\in(0,1)$. Replace $S_P$ with $o\mapsto E_P\left(Y\mid A=1,W(-k)=w(-k)\right)$ from \hyperlink{ex:2}{Example 2}, yielding then \[ D_P^S(o)=\frac{a}{p}\left[y-E_{P}\left(Y\mid A,W(-k)=w(-k)\right)\right]\ .\] It then follows that $D_0^R\not\equiv D_0^S$ and in particular $\Gamma_0$ is no longer constant. In this case, the limit distribution given in \autoref{thm:asympt} is non-degenerate. Consistent estimation of $q_{1-\alpha}$ thus yields a test that asymptotically controls type I error. Given that the proposed estimator $\psi_n$ converges to zero faster than $n^{-1}$,  the probability of rejecting the null approaches zero as sample size grows. In principle, we could have chosen any positive cutoff given that $n\psi_n\rightarrow 0$ in probability, but choosing a more principled cutoff seems judicious.

Because $p$ is known, the remainder term $\Rem_{P}^S$ is equal to zero. Furthermore, in view of the independence between $A$ and all other variables, one can estimate $E_{P_0}\left(Y\mid A=0,W(-k)\right)$ by regressing $Y$ on $W(-k)$ using all of the data without including the covariate $A$.

In future work, it may also be worth checking to see if the parameter is third-order differentiable under the null, and if so whether or not this allows us to construct an $\alpha$-level test without resorting to an artificial source of randomness.

\section{Simulation studies} \label{sec:num}

In simulation studies, we have explored the performance of our proposed test in the context of Examples 1, 2 and 3, and have also compared our method to the approach of \cite{Racineetal2006} for which software is readily available -- see, e.g., the \texttt{R} package \texttt{np} \citep{Hayfield&Racine2008}. We report the results of our simulation studies in this section.

\subsection{Simulation scenario 1}
We use an observed data structure $(W,A,Y)$, where $W\triangleq (W_1,W_2,\ldots,W_5)$ is drawn from a standard 5-dimensional normal distribution, $A$ is drawn according to a $\textnormal{Bernoulli}(0.5)$ distribution, and $Y=\mu(A,W) + 5\xi(A,W)$, where the different forms of the conditional mean function $\mu(a,w)$ are given in \tabref{tab:EYAW}, and $\xi(a,w)$ is a random variate following a Beta distribution with shape parameters $\alpha=3\expit(aw_2)$ and $\beta=2\expit[(1-a)w_1]$ shifted to have mean zero, where $\expit(x)=1/(1+\exp(-x))$.

We performed tests of the null in which $\mu(1,W)$ is equal to $\mu(0,W)$ almost surely and in distribution, as presented in Examples 1 and 2, respectively. Our estimate $\hat{P}_n$ of $P_0$ was constructed using the knowledge that $P_0\left(A=1\mid W\right)=1/2$, as would be available, for example, in the context of a randomized trial. The conditional mean function $\mu(a,w)$ was estimated using the ensemble learning algorithm Super Learner \citep{vanderLaan&Polley&Hubbard07}, as implemented in the \texttt{SuperLearner} package \citep{SuperLearner2013}. This algorithm was implemented using $10$-fold cross-validation to determine the best convex combination of regression function candidates minimizing mean-squared error using a candidate library consisting of \texttt{SL.rpart}, \texttt{SL.glm.interaction}, \texttt{SL.glm}, \texttt{SL.earth}, and \texttt{SL.nnet}. We used the results of \autoref{cor:Seq0} to evaluate significance for \hyperlink{ex:1}{Example 1}, and the eigenvalue approach presented in \secref{sec:estcutoff} to evaluate significance for \hyperlink{ex:2}{Example 2}, where we used all of the positive eigenvalues for $n=125$ and the largest $200$ positive eigenvalues for $n>125$ using the \texttt{rARPACK} package \citep{Qiu&May2014}.

We ran 1,000 Monte Carlo simulations with samples of size $125$, $250$, $500$, $1000$, and $2000$, except for the \texttt{np} package, which we only ran for $500$ Monte Carlo simulations due to its burdensome computation time. For \hyperlink{ex:1}{Example 1} we compared our approach with that of \cite{Racineetal2006} using the \texttt{npsigtest} function from the \texttt{np} package. This requires first selecting a bandwidth, which we did using the \texttt{npregbw} function, specifying that we wanted a local linear estimator and the bandwidth to be selected using the \texttt{cv.aic} method \citep{Hayfield&Racine2008}.

\begin{table}
\caption{\label{tab:EYAW}Conditional mean function in each of three simulation settings within simulation scenario 1. Here,  $m(a,w)\triangleq 0.2\left(w_1^2 + w_2 - 2w_3w_4\right)$, and the third and fourth columns indicate, respectively, whether $\mu(1,W)$ and $\mu(0,W)$ are equal in distribution or almost surely.}
\centering
\begin{tabular}{l l c c}
\hline
 & $\mu(a,w)$ & $\eqd$ & ${\buildrel {a.s.} \over =}$ \\
\hline\hline
Simulation 1a & $m(a,w)$ & $\times$ & $\times$ \\
\hline
Simulation 1b & $m(a,w) + 0.4[aw_3 + (1-a)w_4]$ & $\times$ &  \\
\hline
Simulation 1c & $m(a,w) + 0.8aw_3$ &  &  \\
\hline
\end{tabular}
\end{table}

\begin{figure}
	\centering
	\includegraphics[width=0.6\textwidth]{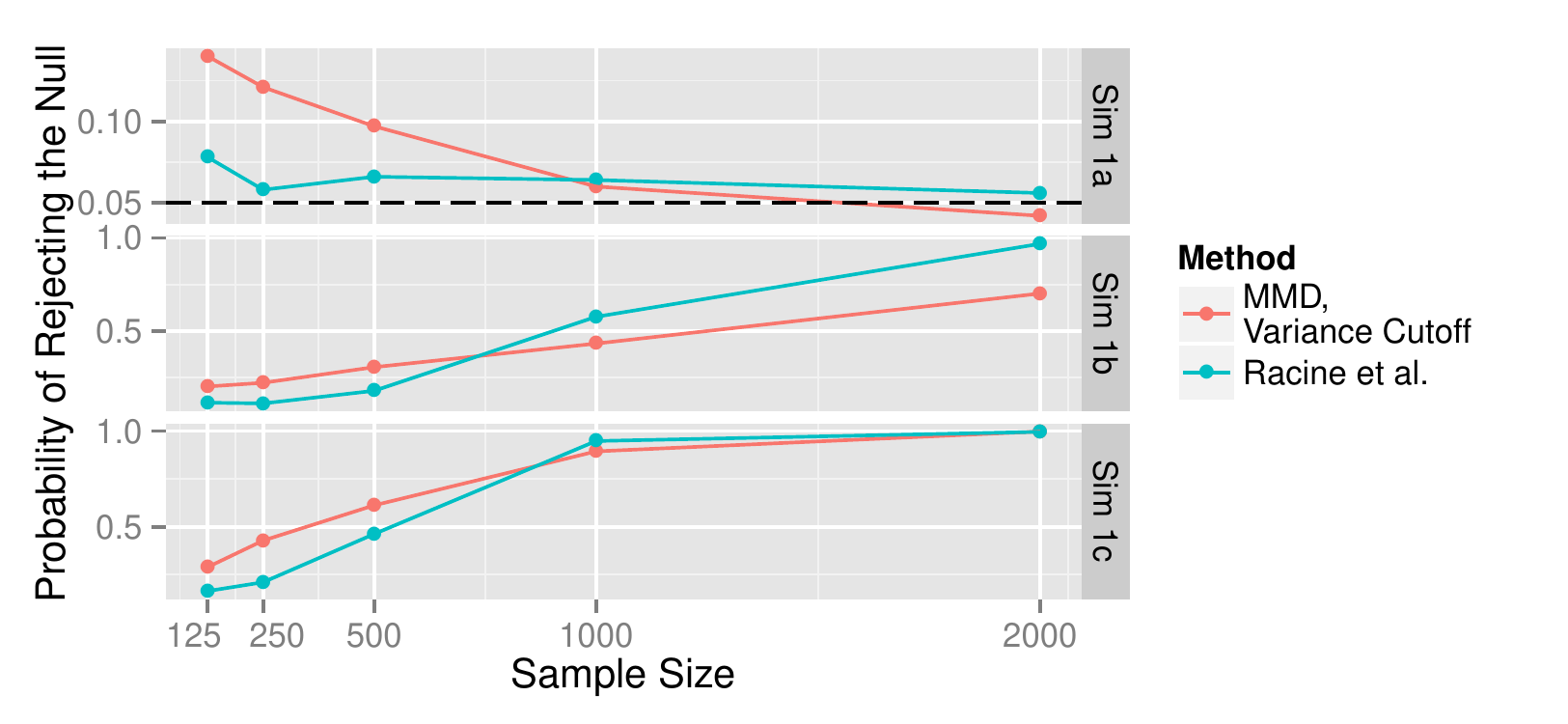}
	\caption{\label{fig:sim1prob}Empirical probability of rejecting the null when testing the null hypothesis that $\mu(1,W)-\mu(0,W)$ is almost surely equal to zero (Example 1) in Simulation 1. \tabref{tab:EYAW} indicates that the null is true in Simulation 1a, and the alternative is true in Simulations 1b and 1c.}
\end{figure}

\figref{fig:sim1prob} displays the empirical coverage of our approach as well as that resulting from use of the \texttt{np} package. At smaller sample sizes, our method does not appear to control type I error near the nominal level. This is likely because we use an asymptotic result to compute the cutoff, even when the sample size is small. Nevertheless, as sample size grows, the type I error of our test approaches the nominal level. We note that in \cite{Racineetal2006}, unlike in our proposal, the bootstrap was used to evaluate the significance of the proposed test. It will be interesting to see if applying a bootstrap procedure at smaller sample sizes improves our small-sample results. At larger sample sizes, it appears that the method of \citeauthor{Racineetal2006} slightly outperforms our approach in terms of power in simulation scenarios 1a and 1b.

In \figref{fig:sim1distr}, the empirical null rejection probability of our test is displayed for simulation scenarios 1a, 1b and 1c. In particular, we observe that our method is able to control type I error for Simulations 1a and 1b when testing the hypothesis that $\mu(1,W)$ is equal in distribution to $\mu(0,W)$. Also, the power of our test increases with sample size, as one would expect. We are not aware of any other test devised for this hypothesis.

\begin{figure}
	\centering
	\includegraphics[width=0.6\textwidth]{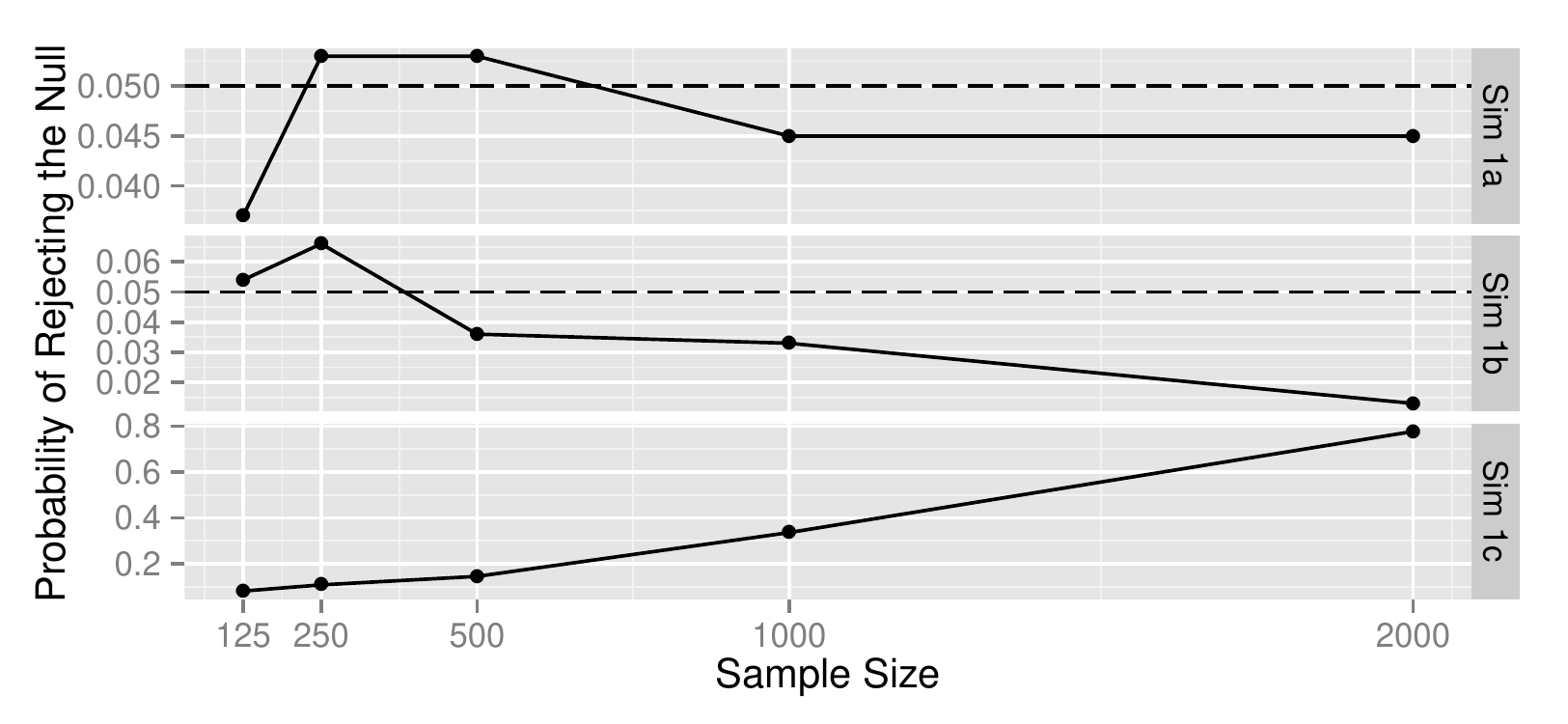}
	\caption{\label{fig:sim1distr}Empirical probability of rejecting the null when testing the null hypothesis that $\mu(1,W)$ is equal in distribution to $\mu(0,W)$ (Example 2) in Simulation 1. \tabref{tab:EYAW} indicates that the null is true in Simulations 1a and 1b, and the alternative is true in Simulation 1c.}
\end{figure}

\subsection{Simulation scenario 2: comparison with \cite{Racineetal2006}}
We reproduced a simulation study from Section 4.1 of \cite{Racineetal2006} at sample size $n=100$. In particular, we let $Y=1 +\beta A(1+W_2^2) + W_1 + W_2 + \epsilon$,
where $A$, $W_1$, and $W_2$ are drawn independently from $\textnormal{Bernoulli}(0.5)$, $\textnormal{Bernoulli}(0.5)$, and $N(0,1)$ distributions, respectively. The error term $\epsilon$ is unobserved and drawn from a $N(0,1)$ distribution independently of all observed variables. The parameter $\beta$ was varied over values $-0.5,-0.4,...,0.4,0.5$ to achieve a range of distributions. The goal is to test whether $E_{0}\left(Y\mid A,W\right)=E_{0}\left(Y\mid W\right)$ almost surely, or equivalently, that $\mu(1,W)-\mu(0,W)=0$ almost surely.

Due to computational constraints, we only ran the `Bootstrap I test' to evaluate significance of the method of \cite{Racineetal2006}. As the authors report, this method is anticonservative relative to their `Bootstrap II test' and indeed achieves lower power (but proper type I error control) in their simulations.

Except for two minor modifications, our implementation of the method in \hyperlink{ex:1}{Example 1} is similar to that as for Simulation 1. For a fair comparison with \cite{Racineetal2006}, in this simulation study, we estimated $P_0\left(A=1\mid W\right)$ rather than treating it as known. We did this using the same Super Learner library and the `\texttt{family=binomial}' setting to account for the fact that $A$ is binary. We also scaled the function $\mu(1,w)-\mu(0,w)$ by a factor of $5$ to ensure most of the probability mass of $R_0$ falls between $-1$ and $1$  (around $99\%$ when $\beta=0$) -- this is equivalent to selecting a bandwidth of $1/5$ for the Gaussian kernel in the definition of the MMD. We also considered a bandwidth of $5/2$: the results were essentially identical and therefore omitted here. We note that even with scaling the variable $Y$ is not bounded as our regularity conditions require. Nonetheless, an evaluation of our method under violations of our assumptions can itself be very informative.

\begin{figure}
	\centering
	\includegraphics[width=0.6\textwidth]{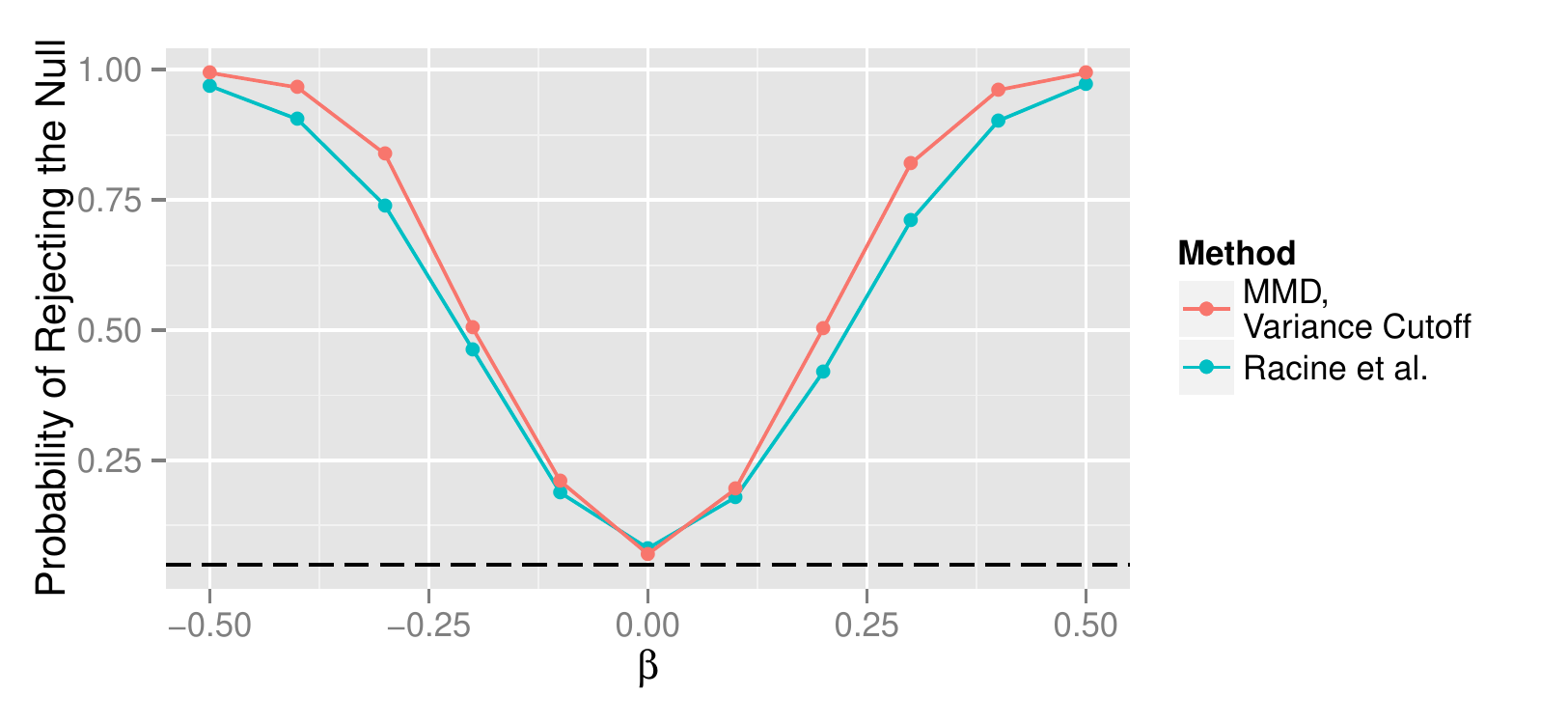}
	\caption{\label{fig:sim2}Empirical probability of rejecting the null when testing the null hypothesis that $\mu(1,W)-\mu(0,W)$ is almost surely equal to zero in Simulation 2.}
\end{figure}

\figref{fig:sim2} displays the empirical null rejection probability of our test as well as that of \cite{Racineetal2006}. In this setup, used by the authors themselves to showcase their test procedure, our method outperforms their proposal, both in terms of type I error control and power.

\subsection{Simulation scenario 3: higher dimensions}
We also explored the performance of our method as extended to tackle higher-dimensional hypotheses, as discussed in \secref{sec:concl}. To do this, we used the same distribution as for Simulation 1 but with $Y$ now a 20-dimensional random variable. Our objective here was to test $\mu(1,W)-\mu(0,W)$ is equal to $(0,0,\ldots,0)$ in probability, where $\mu(a,w)\triangleq (\mu_1(a,w),\mu_2(a,w),\ldots,\mu_{20}(a,w))$ with $\mu_j(a,w)\triangleq E_0\left(Y_j\mid A=a,W=w\right)$. Conditional on $A$ and $W$, the coordinates of $Y$ are independent. We varied the number of coordinates that represent signal and noise. For signal coordinate $j$, given $A$ and $W$, $20Y_j$ was drawn from the same conditional distribution as $Y$ give $A$ and $W$ in Simulation 1c. For noise coordinate $j$, given $A$ and $W$, $20Y_j$ was drawn from the same conditional distribution as $Y$ given $A$ and $W$ in Simulation 1a.

Relative to Simulation 1, we have scaled each coordinate of the outcome to be one twentieth the size of the outcome in Simulation 1.  Apart from the Gaussian kernel with bandwidth one, which we have adopted throughout this paper, we considered defining the MMD with a Gaussian kernel with bandwidth $1/2$. Alternatively, this could be viewed as considering bandwidths $1/20$ and $1/40$ if the outcome had not been scaled by $1/20$.

We ran the same Super Learner to estimate $\mu(1,w)$ as in Simulation 1, and we again treated the probability of treatment given covariates as known. We evaluated significance by estimating all of the positive eigenvalues of the centered Gram matrix for $n=125$ and the largest $200$ positive eigenvalues of the centered Gram matrix for $n> 125$.

\begin{figure}
	\centering
	\includegraphics[width=0.8\textwidth]{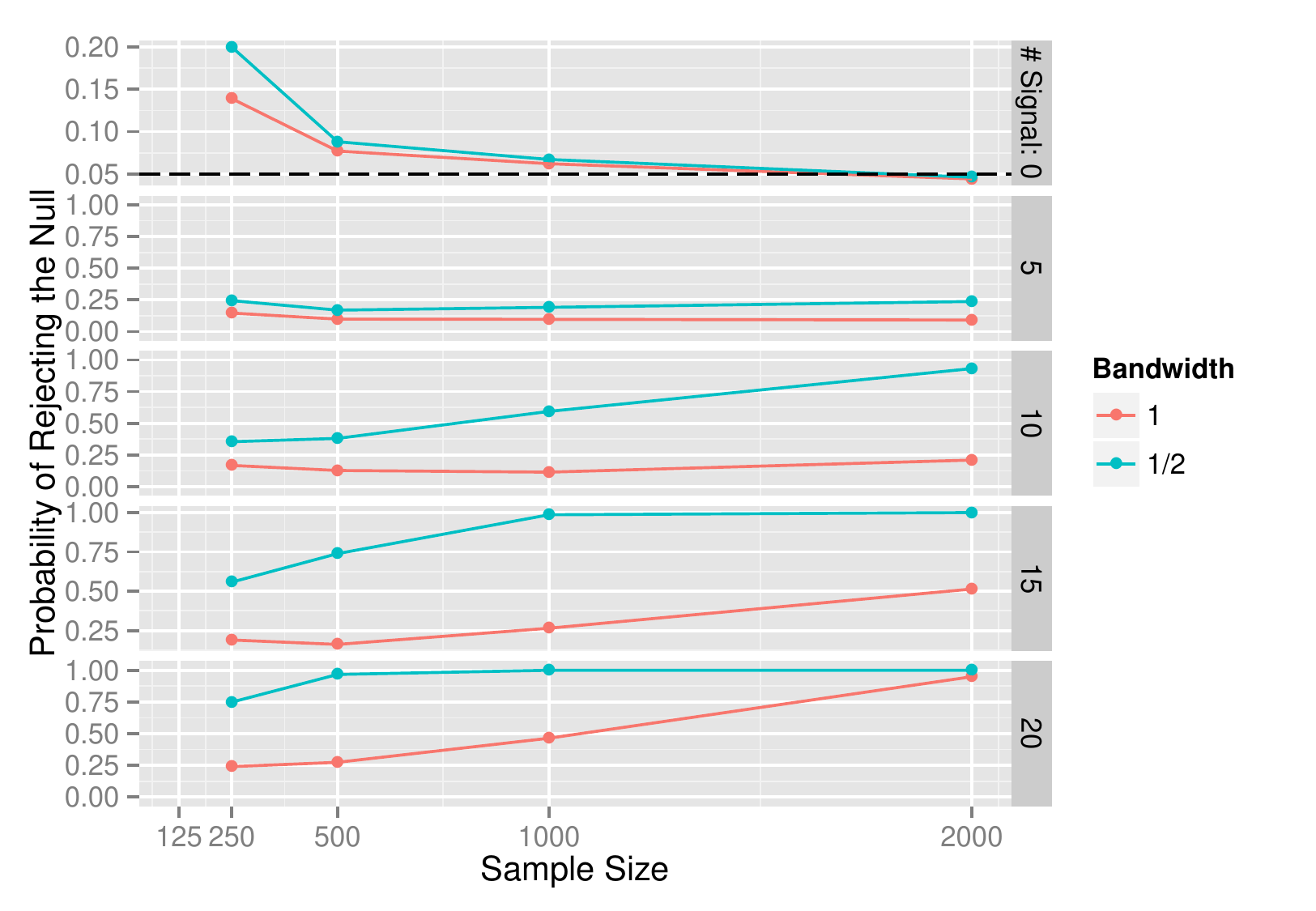}
	\caption{\label{fig:sim3}Probability of rejecting the null when testing the null hypothesis that $\mu(1,W)-\mu(0,W)$ is almost surely equal to zero in Simulation 3.}
\end{figure}

In \figref{fig:sim3}, the empirical null rejection probability is displayed for our proposed MMD method based on bandwidths 1 and 1/2. Our proposal appears to control type I error well at moderate to large sample sizes (i.e., $n\ge 500$). We did not include the results for sample size $125$ in the figure because type I error control was too poor. In particular, for zero signal coordinates, the probability of rejection was $0.24$ for bandwidth $1$ and $0.33$ for bandwidth $1/2$. For a signal of $5$, the empirical probability of rejection decreases between a sample size of $250$ and $500$, likely due to the poor type I error control at sample size $250$. Nonetheless, this simulation shows that, overall, our method indeed has increasing power as sample size grows or as the number of coordinates $j$ for which $\mu_j(1,W)-\mu_j(0,W)$ not equal to zero in probability increases. This figure also highlights that the bandwidth may be an important determinant of finite-sample power, therefore warranting further scrutiny in future work.

\section{Concluding remarks} \label{sec:concl}

We have presented a novel approach to test whether two unknown functions are equal in distribution. Our proposal explicitly allows, and indeed encourages, the use of flexible, data-adaptive techniques for estimating these unknown functions as an intermediate step. Our approach is centered upon the notion of maximum mean discrepancy, as introduced in \cite{Grettonetal2006}, since the MMD provides an elegant means of contrasting the distributions of these two unknown quantities. In their original paper, these authors showed that the MMD, which in their context tests whether two probability distributions are equal using $n$ random draws from each distribution, can be estimated using a $U$- or $V$-statistic. Under the null hypothesis, this $U$- or $V$-statistic is degenerate and converges to the true parameter value quickly. Under the alternative, it converges at the standard $n^{-1/2}$ rate. Because this parameter is a mean over a product distribution from which the data were observed, it is not surprising that a $U$- or $V$-statistic yields a good estimate of the MMD. What is surprising is that we were able to construct an estimator with these same rates even when the null hypothesis involves unknown functions that can only be estimated at slower rates. To accomplish this, we used recent developments from the higher-order pathwise differentiability literature. This appears to be the first use of these developments to address an open methodological problem. Our simulation studies indicate that our asymptotic results are meaningful in finite samples, and that in specific examples for which other methods exist, our methods generally perform at least as well as these established, tailor-made methods. Of course, the great appeal of our proposal is that it applies to a much wider class of problems.

We conclude with several possible extensions of our method that may increase further its applicability and appeal.
\begin{enumerate}
\item Although this condition is satisfied in all but one of our examples, requiring $R$ and $S$ to be in $\mathscr{S}$ can be somewhat restrictive. Nevertheless, it appears that this condition may be weakened by instead requiring membership to $\mathscr{S}^*$, the class of all parameters $T$ for which there exist some $M<\infty$ and elements $T^1,T^2,\ldots,T^M$ in $\mathscr{S}$ such that $T=\sum_{m=1}^{M}T^m$. While the results in our paper can be established in a similar manner for functions in this generalized class, the expressions for the involved gradients are quite a bit more complicated. Specifically, we find that, for $T,U\in\mathscr{S}^*$ with $T=\sum_{m=1}^{M}T^m$ and $U=\sum_{\ell=1}^{L}U^\ell$, the quantity $\Gamma_P^{TU}(o_1,o_2)$ equals
\begin{align*}
&e^{-[T_P(o_1)-U_P(o_2)]^2} \\
&+\sum_{\ell=1}^L E_{P}\left\{2\left[T_P(o_1)-U_P(O)\right]e^{-[T_P(o_1)-U_P(O)]^2}\middle|X^{U^\ell}=x_2^{U^\ell}\right\} D_P^{U^\ell}(o_2) \\
&\ \ -\sum_{m=1}^M E_{P}\left\{2\left[T_P(O)-U_P(o_2)\right]e^{-[T_P(O)-U_P(o_2)]^2}\middle|X^{T^m}=x_1^{T^m}\right\} D_P^{T^m}(o_1) \\
&\ \ -\sum_{\ell=1}^L \sum_{m=1}^M E_{P^2}\Big[\left\{4\left[T_P(O_1)-U_P(O_2)\right]^2-2\right\}e^{-[T_P(O_1)-U_P(O_2)]^2} \\
&\hspace{1in}\Big|X_1^{T^\ell}=x_1^{T^\ell},X_2^{U^m}=x_2^{U^m}\Big]D_P^{T^\ell}(o_1) D_P^{U^m}(o_2)\ .
\end{align*}
In particular, we note the need for conditional expectations with respect to $X^{R^m}$ and $X^{S^m}$ in the definition of $\Gamma$, which could render the implementation of our method more difficult. While we believe this extension is promising, its practicality remains to be investigated.

\item While our paper focuses on univariate hypotheses, our results can be generalized to higher dimensions. Suppose that $P\mapsto R_P$ and $P\mapsto S_P$ are $\mathbb{R}^d$-valued functions on $\mathscr{O}$. The class $\mathscr{S}_d$ of allowed such parameters can be defined similarly as $\mathscr{S}$, with all original conditions applying componentwise. The MMD for the vector-valued parameters $R$ and $S$ using the Gaussian kernel is given by $\Psi_d(P)\triangleq \Phi_d^{RR}(P)-2\Phi_d^{RS}(P)+\Phi_d^{SS}(P)$, where for any $T,U\in\mathscr{S}_d$ we set \[\Phi_d^{TU}(P)\triangleq \iint e^{-\|T_P(o_1)-U_P(o_2)\|^2}dP(o_1)dP(o_2)\ .\] It is not difficult to show then that, for any $T,U\in\mathscr{S}_d(P_0)$, $\Gamma_{d,P}^{TU}(o_1,o_2)$ is given by 
\begin{align*}
\Big[2&\left[T_P(o_1)-U_P(o_2)\right]'\left[D_P^U(o_2) -D_P^T(o_1)\right] + 1 \\
&\;- 2D_P^T(o_1)'\left\{2\left[T_P(o_1)-U_P(o_2)\right] \left[T_P(o_1)-U_P(o_2)\right]'-\Id\right\}D_P^U(o_2)\Big] \\
&\;\times e^{-\norm{T_P(o_1)-U_P(o_2)}^2},
\end{align*}
where $\Id$ denotes the $d$-dimensional identity matrix and $A'$ denotes the transpose of a given vector $A$. Using these objects, the method and results presented in this paper can be replicated in higher dimensions rather easily.

\item Our results can be used to develop confidence sets for infinite dimensional parameters by test inversion. Consider a parameter $T$ satisfying our conditions. Then one can test if $R_0\triangleq T_0-f$ is equal in distribution to zero for any fixed function $f$ that does not rely on $P$. Under the conditions given in this paper, a $1-\alpha$ confidence set for $T_0$ is given by all functions $f$ for which we do not reject $\mathscr{H}_0$ at level $\alpha$. The blip function from \hyperlink{ex:1}{Example 1} is a particularly interesting example, since a confidence set for this parameter can be mapped into a confidence set for the sign of the blip function, i.e. the optimal individualized treatment strategy \citep{Robins04}. We would hope that the omnibus nature of the test implies that the confidence set does not contain functions $f$ that are ``far away'' from $T_0$, contrary to a test which has no power against certain alternatives. Formalization of this claim is an area of future research.

\item To improve upon our proposal for nonparametrically testing variable importance via the conditional mean function, as discussed in \secref{sec:illustr}, it may be fruitful to consider the related Hilbert Schmidt independence criterion \citep{Grettonetal2005}. Higher-order pathwise differentiability may prove useful to estimate and infer about this discrepancy measure.

\end{enumerate}

\section*{Acknowledgement}
The authors thank Noah Simon for helpful discussions. Alex Luedtke was supported by the Department of Defense (DoD) through the National Defense Science \& Engineering Graduate Fellowship (NDSEG) Program. Marco Carone was supported by a Genentech Endowed Professorship at the University of Washington. Mark van der Laan was supported by NIH grant R01 AI074345-06.

\setcounter{equation}{0}
\renewcommand{\theequation}{A.\arabic{equation}}

\newtheorem{apptheorem}{Theorem}
\renewcommand{\theapptheorem}{A.\arabic{apptheorem}}
\newtheorem{appcorollary}{Corollary}
\renewcommand{\theappcorollary}{A.\arabic{appcorollary}}
\newtheorem{applemma}{Lemma}
\renewcommand{\theapplemma}{A.\arabic{applemma}}


\section*{Appendix A: Pathwise differentiability} \label{sec:pd}
We now review first- and second-order pathwise differentiability. Define the following fluctuation submodel through $P_0$:
\begin{align*}
dP_t(o)&\triangleq \left(1 + th_1(o) + t^2 h_2(o)\right)dP_0(o), \\
&\mbox{ where }P_0 h_j = 0\textrm{ and }\sup_{o\in\mathcal{O}}|h_j(o)|<\infty, j=1,2.
\end{align*}
The function $h_1$ is a score, and the closure of the linear span of all such scores yields the unrestricted tangent space $L_0^2(P_0)$, i.e. the set of $P_0$ mean zero functions in $L^2(P_0)$. Note that it is the resulting (first-order) tangent space that is important, as all differentiability properties discussed in this appendix are equivalent for any set of functions $h_1$ that yield the same tangent space. Hence, the restriction that $\sup_{o\in\mathcal{O}}|h_1(o)|<\infty$, while convenient, will have no impact on the resulting differentiability properties. The second-order tangent space is also determined by the first-order tangent space (see \citealp{Caroneetal2014} and the references therein).

Let $\psi_t\triangleq \Psi(P_t)$. The parameter $\Psi$ is called (first-order) pathwise differentiable at $P_0$ if there exists a $D_1^{\Psi}\in L_0^2(P_0)$ such that
\begin{align*}
\psi_t-\psi_0&= t P_0 D_1^{\Psi} h_1 + o(t).
\end{align*}
We call $D_1^{\Psi}$ the first-order canonical gradient of $\Psi$ at $P_0$, where we note that $D_1^\Psi(O)$ is almost surely unique because $\mathcal{M}$ is nonparametric. The canonical gradient $D_1^{\Psi}$ depends on $P_0$ but this is omitted in the notation because we will only discuss pathwise differentiability at $P_0$.

A function $f : \mathcal{O}^2\rightarrow\mathbb{R}$ is called ($P$) one-degenerate if it is symmetric and $P f(o,\cdot)=0$. We will use the notation $P^2 f = E_{P^2}[f(O_1,O_2)]$. The parameter $\Psi$ is called second-order pathwise differentiable at $P_0$ if there exists some symmetric, one-degenerate, $P_0^2$ square integrable function $D_2^{\Psi}$ such that
\begin{align*}
\psi_t-\psi_0=& t P_0 D_1^{\Psi} h_1 + \frac{1}{2}t^2 P_0 D_1^{\Psi}h_2 + \frac{1}{2}t^2\int\int D_2^{\Psi}(o_1,o_2)h_1(o_1)h_1(o_2)dP_0(o_2)dP_0(o_1) + o(t^2).
\end{align*}

\section*{Appendix B: Empirical Process Results} \label{app:bdUproc}
We now present two results from empirical process theory, the first of which can be used to control the $U$-processes that we deal with in the main text when $\mathscr{H}_0$ holds, and the second of which can be used to establish an empirical process condition that is used when $\mathscr{H}_1$ holds.

Before giving an overview of the empirical process theory that we use, we review the notion of a covering number. Let $\mu$ be a probability measure over $\mathcal{Z}$. For a class of functions $f : \mathcal{Z}\rightarrow \mathbb{R}$ with envelope $F$ (i.e., $|f(z)|\le F(z)$ for all $z\in\mathcal{Z}$), where $0<\norm{F}_{2,\mu}<\infty$, define the covering number $N(\epsilon,\mu,\mathcal{F},F)$ as the cardinality of the smallest subcollection $\mathcal{F}^*\subseteq\mathcal{F}$ such that, for all $f\in\mathcal{F}$, $\min_{f^*\in\mathcal{F}^*}\norm{f-f^*}_{2,\mu}\le \epsilon \norm{F}_{2,\mu}$.

\subsection*{Appendix B.1: Bounding $U$-processes}
When $\mathscr{H}_0$ holds, our proofs rely on $\Un(\tilde{\Gamma}_n-\Gamma_0)=o_{P_0}(n^{-1})$ for our method to control the type I error rate. This rate turns out to be plausible, but requires techniques which are different from the now classical empirical process techniques which can be used to establish that $(P_n-P_0)(f_n-f_0)=o_{P_0}(n^{-1/2})$ provided $P_0(f_n-f_0)^2\rightarrow 0$ in probability.

We ignore measurability concerns in this appendix with the understanding that minor modifications are needed to make these results rigorous.

We remind the reader that a function $g : \mathcal{O}^2\rightarrow\mathbb{R}$ is called one-degenerate if and only if $g$ is symmetric in its arguments and $P_0 g(o,\cdot)=0$ for all $o\in\mathcal{O}$. Let $\mathcal{G}$ denote a collection of one-degenerate functions mapping from $\mathcal{O}^2$ to $\mathbb{R}$, where $\sup_g|g(o_1,o_2)|<G(o_1,o_2)$ for all $o_1,o_2$ and some envelope function $G\in L^2(P_0)$.

Suppose we wish to estimate some $g_0\in \mathcal{G}$. We are given a sequence of estimates $\hat{g}_n\in\mathcal{G}$ that is consistent for $g_0$. Our objective is to show that
\begin{align*}
n\Un (\hat{g}_n-g_0)=o_{P_0}(1).
\end{align*}
The uniform entropy integral of $\mathcal{G}$ is given by
\begin{align}
J(t,\mathcal{G},G)&\triangleq \sup_Q \int_0^t \log N(\epsilon,Q,\mathcal{G},G) d\epsilon, \label{eq:entrint}
\end{align}
where the supremum is over all distributions $Q$ with support $\mathcal{O}^2$ and $\norm{G}_{Q,2}>0$. We note that the above definition of the entropy integral upper bounds the covering integral given by \cite{Nolan&Pollard1987}, which considers a particular choice of $Q$. The entropy integral above lacks the square root around the logarithm in the integral that is seen in the standard definition of the uniform entropy integral used to bound empirical processes \citep[see, e.g.,][]{vanderVaartWellner1996}. 

For each $g\in\mathcal{G}$, let $H_g$ represent the Hilbert-Schmidt operator on $L^2(P_0)$ given by $(H_g f)(o) = P g(o,\cdot)f(\cdot)$. Let $\{W_j : j=1,2,...\}$ be a sequence of i.i.d. standard normal random variables and $\{w_j : j=1,2,...\}$ be an orthonormal basis of $L^2(P_0)$. Let $\tilde{Q}$ be a process on $\mathcal{G}$ defined by
\begin{align*}
\tilde{Q}(g)=\sum_{j=1}^\infty \langle H_g w_j,w_j\rangle(W_j^2-1) + \sum_{i\not=j} \langle H_g w_j,w_i\rangle W_i W_j.
\end{align*}
A functional $M : \mathcal{G}\rightarrow\mathbb{R}$ is said to belong to $C(\mathcal{G},P_0^2)$ if $g\mapsto M(g)$ is uniformly continuous for the $L^2(P_0^2)$ seminorm and $\sup_{\mathcal{G}}|M(g)|<\infty$.

We have modified the statement from \cite{Nolan&Pollard1988} slightly to apply to the entropy integral given in (\ref{eq:entrint}). We omit an analogue to condition (ii) from \citeauthor{Nolan&Pollard1988}'s statement of the theorem below because it is implied by our strengthening of their condition (i).
\begin{apptheorem}[Theorem 7, \citealp{Nolan&Pollard1988}] \label{thm:nolanpollard2}
Suppose that the one-degenerate class $\mathcal{G}$ satisfies
\begin{enumerate}[label=(\roman*')]
	\item $J(1,\mathcal{G},G)<\infty$; \label{it:nolanpollardi}
	\stepcounter{enumi}
	\item $\sup_Q \log N(\epsilon,Q\times P_0,\mathcal{G},G)<\infty$ for each $\epsilon>0$, where the supremum is over distributions $Q$ with support $\mathcal{O}$. \label{it:nolanpollardiii}
\end{enumerate}
Then there is a version of $\tilde{Q}$ with continuous sample paths in $C(\mathcal{G},P_0\times P_0)$ and $n\mathbb{U}_n$ converges in distribution to $\tilde{Q}$.
\end{apptheorem}

We will use the following corollary to control the cross-terms.
\begin{appcorollary}
Suppose that $\mathcal{G}$ satisfies the conditions of \autoref{thm:nolanpollard2} and $\hat{g}_n$ is a sequence of one-degenerate random functions that take their values in $\mathcal{G}$ such that $P_0^2(\hat{g}_n-g_0)^2\rightarrow 0$ in probability for some $g_0\in \mathcal{G}$. Then $n\mathbb{U}_n(\hat{g}_n-g_0)\rightarrow 0$ in probability. 
\end{appcorollary}
The proof relies on the continuity of sample paths of (a version of) $\tilde{Q}$. The proof is omitted, but we refer the reader to the proof of Lemma 19.24 in \cite{vanderVaart98} for the analogous empirical process result.

\subsection*{Appendix B.2: Controlling $\int \Gamma_n(\cdot,o)dP_0(o)$}
We now give sufficient conditions under which $o\mapsto \int \Gamma_n(o_1,o)dP_0(o)$ belongs to a fixed Donsker class with probability approaching one. We recall from \cite{vanderVaartWellner1996} that a class $\mathcal{F}$ of functions mapping from $\mathscr{O}$ to $\mathbb{R}$ is Donsker if its uniform entropy integral is finite, which holds if its covering number grows sufficiently slowly as the approximation becomes arbitrarily precise.

Let $\mathcal{G}_2$ be some class of functions $g : \mathscr{O}^2\rightarrow[-M,M]$, $M<\infty$, that contains $\{(o_1,o_2)\mapsto \Gamma_n(o_1,o_2) : \Gamma_n\}$. Without loss of generality, we suppose that $M=1$. We take the constant function $G_2\equiv 1$ as envelope for $\mathcal{G}_2$. Let $\mathcal{G}_1\triangleq \{o_1\mapsto \int g_2(o_1,o_2)dP_0(o_2): g_2\in\mathcal{G}_2\}$, and note that this class similarly has envelope $G_1\equiv 1$. The main observation of this subappendix is that
\begin{align}
\sup_{Q_1} N(\epsilon,Q_1,\mathcal{G}_1,G_1)\le \sup_{Q_2} N(\epsilon,Q_2,\mathcal{G}_2,G_2),\label{eq:coveringnums}
\end{align}
where the supremum on the left is over all distributions $Q_1$ on $\mathscr{O}$ such that $\norm{G_1}_{2,Q_1}>0$ and the supremum on the right is over all distributions $Q_2$ on $\mathscr{O}^2$ such that $\norm{G}_{2,Q_2}>0$. If we can show this, then the uniform entropy integrals are also ordered \citep[Section~2.5 in][]{vanderVaartWellner1996}:
\begin{align}
\int_0^\infty \sup_{Q_1} \sqrt{\log N(\epsilon,Q_1,\mathcal{G}_1,G_1)} d\epsilon\le \int_0^\infty \sup_{Q_2} \sqrt{\log N(\epsilon,Q_2,\mathcal{G}_2,G_2)} d\epsilon, \label{eq:entropyints}
\end{align}
where the left- and right-hand sides are the uniform entropy integrals of $\mathcal{G}_1$ and $\mathcal{G}_2$, respectively. Hence, it will suffice to show that the right-hand side is finite. This can be accomplished using the variety of techniques given in Chapter 2 of \cite{vanderVaartWellner1996}. 

We now establish \eqref{eq:coveringnums}. Fix a measure $Q_1$ over $\mathscr{O}$. Let $Q_2$ represent the product measure $Q_1\times P_0$.
Fix $\epsilon>0$. Let $g_{2,1},\ldots,g_{2,m}$ be an $\epsilon\norm{G_2}_{2,Q_2}$ cover of $\mathcal{G}_2$ under $\norm{\cdot}_{2,Q_2}$ so that $\min_j \norm{g_2-g_{2,j}}_{2,Q_2}<\epsilon\norm{G_2}_{2,Q_2}$, where we take $m$ to be equal to its minimal possible value $N(\epsilon,Q_2,\mathcal{G}_2,G_2)$. For each $j$, let $g_{1,j}\equiv \int g_{2,j}(o_1,o_2) dP_0(o_2)$. Fix $g_1\in \mathcal{G}_1$. Recall that, by the definition of $\mathcal{G}_1$, there exists a $g_2\in\mathcal{G}_2$ such that $g_1(\cdot)=\int g_2(\cdot,o)dP_0(o)$. Let $j^*$ be such that $\norm{g_2-g_{2,j^*}}_{2,Q_2}\le \epsilon\norm{G_2}_{2,Q_2}$ for this $g_2$. Observe that
\begin{align*}
\norm{g_1-g_{1,j^*}}_{2,Q_1}^2
&= \int \left(\int \left[g_2(o_1,o_2)-g_{2,j^*}(o_1,o_2)\right]dP_0(o_2)\right)^2 dQ_1(o_1) \\
&\le \int \left[g_2(o_1,o_2)-g_{2,j^*}(o_1,o_2)\right]^2 dQ_2(o_1,o_2) = \norm{g_2-g_{2,j^*}}_{2,Q_2}^2.
\end{align*}
By the choice of $j^*$, it follows that $\norm{g_1-g_{1,j^*}}_{2,Q_1}\le \epsilon \norm{G_2}_{2,Q_2}=\epsilon \norm{G_1}_{2,Q_1}$, where we used that $G_1\equiv 1$ and $G_2\equiv 1$. That is, $g_{1,1},\ldots,g_{1,m}$ is an $\epsilon \norm{G_1}_{2,Q_1}$ cover of $\mathcal{G}_1$ under $\norm{\cdot}_{2,Q_1}$. Thus, $N(\epsilon,Q_1,\mathcal{G}_1,G_1)\le m$. Recalling that we took $m=N(\epsilon,Q_2,\mathcal{G}_2,G_2)$, we have shown that $N(\epsilon,Q_1,\mathcal{G}_1,G_1)\le N(\epsilon,Q_2,\mathcal{G}_2,G_2)$. As $Q_1$ was arbitrary, for each $Q_1$ with support $\mathscr{O}$ there exists a $Q_2$ with support $\mathscr{O}^2$ such that the the preceding inequality holds. Hence, \eqref{eq:coveringnums} holds, and thus so too does the uniform entropy integral ordering \eqref{eq:entropyints}.

\section*{Appendix C: proofs}

For any $T\in\mathscr{S}$, we will use the shorthand notation $T_t\triangleq  T_{P_t}$, $\left.\frac{d}{dt}T_t\right|_{t=\tilde{t}}\triangleq  \dot{T}_{\tilde{t}}$ and $\left.\frac{d^2}{dt^2}T_t\right|_{t=\tilde{t}}\triangleq  \ddot{T}_{\tilde{t}}$. Throughout the appendix we use the following fluctuation submodel through $P_0$ for pathwise differentiability proofs:
\begin{align}
dP_t(o)&\triangleq \left(1 + th_1(o) + t^2 h_2(o)\right)dP_0(o), \nonumber \\
&\mbox{ where }P_0 h_j = 0\textrm{ and }\sup_{o\in\mathcal{O}}|h_j(o)|<\infty, j=1,2. \label{eq:fluctsubmod}
\end{align}

\subsection*{Proofs for \secref{sec:rep}}
We give two lemmas before proving \autoref{thm:pd1}.
\begin{applemma} \label{lem:phidt}
For any $T,U\in\mathscr{S}$ and any fluctuation submodel $dP_t=\left(1+th_1+t^2h_2\right)dP_0$, we have that, for all $\tilde{t}$ in a neighborhood of zero,
\begin{align*}
\dot{\Phi}^{TU}_{\tilde{t}}\ &=\ \int\left[\int e^{-[T_{\tilde{t}}(x_1^T)-U_{\tilde{t}}(x_2^U)]^2} dP_{\tilde{t}}(x_1^T)\right]\left[h_1(o_2) + 2\tilde{t}h_2(o_2)\right]dP_0(o_2) \\
&\hspace{.2in}+ \int\left[\int e^{-[T_{\tilde{t}}(x_1^T)-U_{\tilde{t}}(x_2^U)]^2} dP_{\tilde{t}}(x_2^U)\right]\left[h_1(o_1) + 2\tilde{t}h_2(o_1)\right]dP_0(o_1) \\
&\hspace{.2in}-2\iint \left[T_{\tilde{t}}(x_1^T)-U_{\tilde{t}}(x_2^U)\right]\left[\left.\frac{d}{dt}T_t(x_1^T)\right|_{t=\tilde{t}}-\left.\frac{d}{dt}U_t(x_2^U)\right|_{t=\tilde{t}}\right] e^{-[T_{\tilde{t}}(x_1^T)-U_{\tilde{t}}(x_2^U)]^2} dP_{\tilde{t}}(x_2^U)dP_{\tilde{t}}(x_1^T).
\end{align*}
\end{applemma}
\begin{proof}[Proof of \autoref{lem:phidt}]
We have that
\begin{align*}
\dot{\Phi}^{TU}_{\tilde{t}}\ &=\ \frac{d}{dt}\left.\iint e^{-[T_t(x_1^T)-U_t(x_2^U)]^2} \left\{\prod_{j=1}^2\left[1+th_1(o_j)+t^2h_2(o_j)\right]\right\}dP_0(o_2)dP_0(o_1)\right|_{t=\tilde{t}} \\
&=\ \iint \frac{d}{dt}\left.e^{-[T_t(x_1^T)-U_t(x_2^U)]^2} \left\{\prod_{j=1}^2\left[1+th_1(o_j)+t^2h_2(o_j)\right]\right\}\right|_{t=\tilde{t}}dP_0(o_2)dP_0(o_1)\ ,
\end{align*}
where the derivative is passed under the integral in view of \ref{it:Sderivsbdd}. The result follows by the chain rule.
\end{proof}

For each $T,U\in\mathscr{S}$, define
\begin{align*}
D^{TU}(o)\ &\triangleq \ -2\Phi^{TU}(P_0) + \int \left\{2\left[U_0(o_1)-T_0(o)\right]D_0^T(o) + 1\right\}  e^{-[T_0(o)-U_0(o_1)]^2} dP_0(o_1) \\
&\hspace{0.5in}+ \int \left\{2\left[T_0(o_1)-U_0(o)\right]D_0^U(o) + 1\right\}  e^{-[T_0(o_1)-U_0(o)]^2} dP_0(o_1)\ .
\end{align*}
We have omitted the dependence of $D^{TU}$ on $P_0$ in the notation. We first give a key lemma about the parameter $\Phi^{TU}$.
\begin{applemma}[First-order canonical gradient of $\Phi^{TU}$] \label{lem:phiuv}
Let $T$ and $U$ be members of $\mathscr{S}$. Then $\Phi^{TU}$ has canonical gradient $D^{TU}$ at $P_0$.
\end{applemma}
\begin{proof}[Proof of \autoref{lem:phiuv}]
To consider first-order behavior it suffices to consider fluctuation submodels in which $h_2(o)=0$ for all $o$. We first derive the first-order pathwise derivative of the parameter $\Phi^{TU}$ at $P_0$. Applying the preceding lemma at $\tilde{t}=0$ yields that
\begin{align*}
\frac{d}{dt}&\Phi^{TU}(P_t)\Big|_{t=0}= \int\left[\int e^{-[T_0(x_1^T)-U_0(x_2^U)]^2} dP_0(x_1^T)\right]h_1(o_2)dP_0(o_2) \\
&+ \int\left[\int e^{-[T_0(x_1^T)-U_0(x_2^U)]^2} dP_0(x_2^U)\right]h_1(o_1)dP_0(o_1) \\
&-2\int\int (T_0(x_1^T)-U_0(x_2^U))(\dot{T}_0(x_1^T)-\dot{U}_0(x_2^U))e^{-[T_0(x_1^T)-U_0(x_2^U)]^2} dP_0(x_2^U)dP_0(x_1^T).
\end{align*}
The first two terms in the last equality are equal to
\begin{align*}
\textrm{First term}&= \int \left(E_{P_0}\left[e^{-[T_0(X^T)-U_0(x^U)]^2}\right]-E_{P_0^2}\left[e^{-[T_0(X_1^T)-U_0(X_2^U)]^2}\right]\right)h_1(o)dP_0(o) \\
\textrm{Second term}&= \int \left(E_{P_0}\left[e^{-[T_0(x^T)-U_0(X^U)]^2}\right]-E_{P_0^2}\left[e^{-[T_0(X_1^T)-U_0(X_2^U)]^2}\right]\right)h_1(o)dP_0(o).
\end{align*}
We now look to find the portion of the canonical gradient given by the third term. We have that
\begin{align*}
-2\int\int&(T_0(x_1^T)-U_0(x_2^U))\dot{T}_0(x_1^T)e^{-[T_0(x_1^T)-U_0(x_2^U)]^2} dP_0(x_2^U)dP_0(x_1^T) \\
&= \int 2E_{P_0}\left[ (U_0(X^U)-T_0(x^T))  e^{-[T_0(x^T)-U_0(X^U)]^2} \right]D_0^T(o)h_1(o)dP_0(o) \\
2\int\int&(T_0(x_1^T)-U_0(x_2^U))\dot{U}_0(x_2^U)e^{-[T_0(x_1^T)-U_0(x_2^U)]^2} dP_0(x_2^U)dP_0(x_1^T) \\
&= \int 2E_{P_0}\left[ (T_0(X^T)-U_0(x^U))  e^{-[T_0(X^T)-U_0(x^U)]^2} \right]D_0^U(o)h_1(o)dP_0(o).
\end{align*}
Collecting terms, a first-order Taylor expansion of $t\mapsto \Phi^{TU}(P_t)$ about $t=0$ yields that
\begin{align*}
\Phi^{TU}(P_t) - \Phi^{TU}(P_0)&= tE_{P_0}\left[D^{TU}(O)h_1(O)\right]+ o(t).
\end{align*}
Thus $\Phi^{TU}$ has canonical gradient $D^{TU}$ at $P_0$.
\end{proof}

The proof of \autoref{thm:pd1} is simple given the above lemma.
\begin{proof}[Proof of \autoref{thm:pd1}]
\autoref{lem:phiuv}, the fact that $\Psi(P)\triangleq \Phi^{RR}(P) - 2\Phi^{RS}(P) + \Phi^{SS}(P)$, and the linearity of differentiation immediately yield that the canonical gradient of $\Psi$ can be written as $D^{RR} - 2 D^{RS} + D^{SS}$. Straightforward calculations show that this is equivalent to $o\mapsto 2[P_0 \Gamma_0(o,\cdot)-\psi_0]$.
\end{proof}

We will use the following lemma in the proof of \autoref{cor:D1degen} to prove that $R_0(O)$ and $S_0(O)$ are degenerate if $D_1^{\Psi}\equiv 0$ and $\mathscr{H}_0$ does not hold. Because we were unable to find the proof that the $U$-statistic kernel for estimating the MMD of two variables $X$ and $Y$ is degenerate if and only if $\mathscr{H}_0$ holds or $X$ and $Y$ are degenerate, we give a proof here that applies in a more general setting than that which we consider in this paper.
\begin{applemma} \label{lem:RSdegen}
Let $Q$ be a distribution over $(X,Y)\in\mathcal{Z}^2$, where $\mathcal{Z}$ is a compact metric space. Let $(x,y)\mapsto k(x,y)$ be a universal kernel on this metric space, i.e. a kernel for which the resulting reproducing kernel Hilbert space $\mathcal{H}$ is dense in the set of continuous funtions on $\mathcal{Z}$ with respect to the supremum metric. Further, suppose that $E_Q\sqrt{k(X,X)}$ and $E_Q\sqrt{k(Y,Y)}$ are finite. Finally, suppose that the marginal distribution of $X$ under $Q$ is different from the marginal distribution of $Y$ under $Q$.

There exists some fixed constant $C$ such that
\begin{align}
\int \langle \phi(x_1)-\phi(y_1),\phi(x_2)-\phi(y_2)\rangle_{\mathcal{H}}  dQ(x_2,y_2)=C \label{eq:canonical}
\end{align}
for ($Q$ almost) all $(x_1,y_1)\in\mathcal{Z}^2$ if and only if the joint distribution of $(X,Y)$ under $Q$ is degenerate at a single point. Above $\langle\cdot,\cdot\rangle_{\mathcal{H}}$ and $\phi(z)\triangleq k(z,\cdot)$ are the inner product and the feature map in $\mathcal{H}$, respectively.
\end{applemma}
\begin{proof}
If $Q$ is degenerate then clearly (\ref{eq:canonical}) holds.

If (\ref{eq:canonical}) holds, then our assumption that $X$ has a different marginal distribution than $Y$ tells us that $C>0$ \citep{Grettonetal2012}. Hence, for almost all $(x_1,y_1)$,
\begin{align*}
\langle\phi(x_1)-\phi(y_1),\mu_X-\mu_Y\rangle_{\mathcal{H}} - \langle\mu_X-\mu_Y,\mu_X-\mu_Y\rangle_{\mathcal{H}}=0,
\end{align*}
where $\mu_X$ and $\mu_Y$ in $\mathcal{H}$ have the property that $\langle\mu_X,f\rangle_{\mathcal{H}}=E_Q f(X)$ and $\langle\mu_Y,f\rangle_{\mathcal{H}}=E_Q f(Y)$ for all $f\in\mathcal{H}$ \citep[Lemma 3 in][]{Grettonetal2012}. The above holds if and only if $\phi(x_1)-\phi(y_1)=\mu_X-\mu_Y$. Noting that $\mu_X-\mu_Y$ does not rely on $x_1,y_1$, it follows that $\phi(x_1)-\phi(y_1)$ must not rely on $x_1,y_1$ for all $(x_1,y_1)$ in some $Q$ probability one set $\mathcal{D}\subseteq\mathcal{Z}^2$.

Fix a continuous function $f : \mathcal{Z}\rightarrow\mathbb{R}$ and $x_1,y_1\in\mathcal{D}$. For any $\epsilon>0$, the universality of $\mathcal{H}$ ensures that there exists an $f_{\epsilon}\in\mathcal{H}$ such that $\norm{f_{\epsilon}-f}_{\infty}\le\epsilon$. By the triangle inequality,
\begin{align*}
\left|f(x_1)-f(y_1) - f_{\epsilon}(x_1)+f_{\epsilon}(y_1)\right|&\le 2\epsilon.
\end{align*}
Because $\phi(x_1)-\phi(y_1)$ is constant and $f\in\mathcal{H}$, $\langle\phi(x_1)-\phi(y_1),f_{\epsilon}\rangle_{\mathcal{H}}=f_{\epsilon}(x_1)-f_{\epsilon}(y_1)$ does not rely on $x_1,y_1$ for any $\epsilon$. Furthermore, the fact that $f_{\epsilon}$ converges to $f$ in supremum norm ensures that $|f_{\epsilon}(x_1)-f_{\epsilon}(y_1)|$ converges to a fixed quantity $K$ (which does not rely on $x_1$ or $y_1$) as $\epsilon\rightarrow 0$. Applying this to the above yields that $f(x_1)-f(y_1)=K$.

As $f$ was an arbitrary continuous function and $X_1\not\equiv Y_1$, we can apply this relation to $z\mapsto z$ and $z\mapsto z^2$ to show that $x_1-y_1$ and $x_1+y_1$ do not rely on the choice of $(x_1,y_1)\in\mathcal{D}$. Hence $(x_1-y_1 + x_1 + y_1)/2=x_1$ and $(x_1 + y_1 - y_1 + x_1)/2=y_1$ do not rely on the choice of $(x_1,y_1)\in\mathcal{D}$. This can only occur if $(x_1,y_1)$ are constant over the probability $1$ set $\mathcal{D}$, i.e. if $Q$ is degenerate.
\end{proof}
For the two-sample problem in \cite{Grettonetal2012}, one can take $Q$ to be a product distribution of the marginal distribution of $X$ and the marginal distribution of $Y$.

\begin{proof}[Proof of \autoref{cor:D1degen}]
We first prove sufficiency. If \hyperlink{it:condH0}{(i)} holds, then $2D^{RS}=D^{RR} + D^{SS}$. It follows that $D_1^{\Psi}\equiv 0$ under $\mathscr{H}_0$. Now suppose  \hyperlink{it:conddegen}{(ii)} holds. It is a simple matter of algebra to verify that $D_1^{RR}\equiv D_1^{RS}\equiv D_1^{SS}\equiv 0$. Hence $D_1^{\Psi}\equiv 0$, yielding the sufficiency of the stated conditions.

We now show the necessity of the stated conditions. Suppose that $\sigma_0=0$ and $\mathscr{H}_0$ does not hold. It is easy to verify that
\begin{align*}
\tilde{D}_1^{\Psi}\triangleq&\; E_{P_0}\left[e^{-[R_0(O)-R_0(o)]^2}\right] + E_{P_0}\left[e^{-[S_0(O)-S_0(o)]^2}\right] \nonumber \\
&-E_{P_0}\left[e^{-[R_0(O)-S_0(o)]^2}\right] - E_{P_0}\left[e^{-[R_0(o)-S_0(O)]^2}\right]-\psi_0
\end{align*}
is a first-order gradient in the model where $R_0$ and $S_0$ are known (possibly an inefficient gradient depending on the form of $R$ and $S$). Call the variance of this gradient $\tilde{\sigma}_0$. As the model where $R_0$ and $S_0$ are known is a submodel of the (locally) nonparametric model, $\tilde{\sigma}_0\le \sigma_0$, and hence $\tilde{\sigma}_0=0$ and $\tilde{D}_1^{\Psi}\equiv 0$. Now, if $\tilde{\sigma}_0=0$ and $\mathscr{H}_0$ does not hold, then \autoref{lem:RSdegen} shows that $R_0(O)$ and $S_0(O)$ are degenerate. Finally, $\tilde{D}_1^{\Psi}\equiv 0$ and the degeneracy of $R_0(O)$ and $S_0(O)$ shows that for almost all $o$,
\begin{align*}
D_1^{\Psi}(o)&= 2D^{RS}(o) = 2(s_0-r_0)\left(D_0^R(o)-D_0^S(o)\right) e^{-[r_0-s_0]^2},
\end{align*}
where we use $r_0$ and $s_0$ to denote the (probability $1$) values of $R_0(O)$ and $S_0(O)$. The above is zero almost surely if and only if $D_0^R\equiv D_0^S$. Thus $\sigma_0=0$ only if \hyperlink{it:condH0}{(i)} or  \hyperlink{it:conddegen}{(ii)} holds.
\end{proof}

We give the following lemma before proving \autoref{thm:pd2}. Before giving the lemma, we define the function $\Pi : \mathscr{S}\rightarrow\mathbb{R}$. Suppressing the dependence on $P_0$ and $h_1$, $h_2$, for all $V\in\mathscr{S}$ and $t\not=0$ we define
\begin{align*}
\Pi(V)&\triangleq 2\int\int \Big[2 (V_0(o_2)-V_0(o_1)) \dot{V}_0(o_2)h_1(o_2) + 2 (V_0(o_2)-V_0(o_1))^2 \dot{V}_0(o_2)^2 \\
&\hspace{5em}+ h_2(o_2) - \dot{V}_0(o_2)^2 + (V_0(o_2)-V_0(o_1))\ddot{V}_0(o_2)\Big]e^{-[V_0(o_2)-V_0(o_1)]^2} dP_0(o_2) dP_0(o_1).
\end{align*}
\begin{applemma} \label{lem:phi2d}
For any fluctuation submodel consistent with (\ref{eq:fluctsubmod}), $T,U\in\mathscr{S}$ with $T_0(O)\eqd U_0(O)$, and $t\in\mathbb{R}$ sufficiently close to zero, we have that
\begin{align*}
\frac{d^2}{dt^2}\Phi^{TU}(P_t)\Big|_{t=0}=&\; 2\int\int \Gamma_0^{TU}(o_1,o_2) h_1(o_1)h_1(o_2)dP_0(o_2)dP_0(o_1) + \Pi(T) + \Pi(U).
\end{align*}
\end{applemma}
\begin{proof}
Let $H_t(o)\triangleq 1 + th_1(o) + t^2h_2(o)$ and $\dot{H}_t(o)\triangleq h_1(o) + 2th_2(o)$.
\begin{align}
\frac{d^2}{dt^2}\Phi^{TU}(P_t)\Big|_{t=0}=&\; \frac{d}{dt}\int\int \Big[H_t(o_1)\dot{H}_t(o_2) + \dot{H}_t(o_1) H_t(o_2) \nonumber \\
&\hspace{5em}-2 (T_{t}(o_1)-U_{t}(o_2))\left(\dot{T}_t(o_1)-\dot{U}_t(o_2)\right) H_t(o_1) H_t(o_2)\Big] \nonumber \\
&\hspace{4em}\times e^{-[T_{t}(o_1)-U_{t}(o_2)]^2} dP_0(o_2)dP_0(o_1)\Big|_{t=0} \label{eq:messyprodrule}
\end{align}
We will pass the derivative inside the integral using \ref{it:Sderivsbdd} and apply the product rule. The first term we need to consider is
\begin{align*}
\frac{d}{dt}&\Big[H_t(o_1)\dot{H}_t(o_2) + \dot{H}_t(o_1) H_t(o_2)-2 (T_{t}(o_1)-U_{t}(o_2))\left(\dot{T}_t(o_1)-\dot{U}_t(o_2)\right) H_t(o_1) H_t(o_2)\Big]\Big|_{t=0} \\
=&\; 2\left[h_2(o_1) + h_1(o_1)h_1(o_2) + h_2(o_2)\right] - 2\left(\dot{T}_0(o_1)-\dot{U}_0(o_2)\right)^2 -2 (T_0(o_1)-U_0(o_2))\left(\ddot{T}_0(o_1)-\ddot{U}_0(o_2)\right) \\
&\hspace{2em} -2 (T_0(o_1)-U_0(o_2))\left(\dot{T}_0(o_1)-\dot{U}_0(o_2)\right) \left(h_1(o_1)+h_1(o_2)\right).
\end{align*}
The second is
\begin{align*}
\left.\frac{d}{dt}e^{-[T_{t}(o_1)-U_{t}(o_2)]^2}\right|_{t=0}&= -2 (T_0(o_1)-U_0(o_2))\left(\dot{T}_0(o_1)-\dot{U}_0(o_2)\right)e^{-[T_0(o_1)-U_0(o_2)]^2}.
\end{align*}
Returning to (\ref{eq:messyprodrule}), this shows that $\frac{d^2}{dt^2}\Phi^{TU}(P_t)\Big|_{t=0}$ is equal to
\begin{align*}
&2\int\int \Big[-2 (T_0(o_1)-U_0(o_2)) \dot{T}_0(o_1)h_1(o_1) + 2 (T_0(o_1)-U_0(o_2))^2 \dot{T}_0(o_1)^2 \\
&\hspace{4.5em} + h_2(o_1) - \dot{T}_0(o_1)^2 - (T_0(o_1)-U_0(o_2))\ddot{T}_0(o_1)\Big] e^{-[T_0(o_1)-U_0(o_2)]^2} dP_0(o_2)dP_0(o_1) \\
&+2\int\int \Big[2 (T_0(o_1)-U_0(o_2)) \dot{U}_0(o_2)h_1(o_2) + 2 (T_0(o_1)-U_0(o_2))^2 \dot{U}_0(o_2)^2 \\
&\hspace{4.5em} + h_2(o_2) - \dot{U}_0(o_2)^2 + (T_0(o_1)-U_0(o_2))\ddot{U}_0(o_2)\Big] e^{-[T_0(o_1)-U_0(o_2)]^2} dP_0(o_2)dP_0(o_1) \\
&+2\int\int \Big[2(T_0(o_1)-U_0(o_2))\left(\dot{U}_0(o_2)h_1(o_1)-\dot{T}_0(o_1)h_1(o_2)\right) \\
&\hspace{4em} - \left(4(T_0(o_1)-U_0(o_2))^2-2\right) \dot{T}_0(o_1) \dot{U}_0(o_2) + h_1(o_1)h_1(o_2)\Big]e^{-[T_0(o_1)-U_0(o_2)]^2} dP_0(o_2)dP_0(o_1).
\end{align*}
The expression inside the second pair of integrals only depends on $o_1$ through $T(o_1)$. Thus we can rewrite this term as $E_{P_0}[f(T(O_1))]$ for a fixed function $f$ that relies on $P_0$, $h_1$, $h_2$, and $U$. Under $\mathscr{H}_0$, we can rewrite this term as $E_{P_0}[f(U(O_1))]$. That is, we can replace each $T(O_1)$ in the second pair of integrals with $U(O_1)$. This yields $\Pi(U)$. Switching the roles of $o_1$ and $o_2$ in the first pair of integrals above and applying Fubini's theorem shows that
\begin{align*}
2\int\int \Big[&2 (T_0(o_2)-U_0(o_1)) \dot{T}_0(o_2)h_1(o_2) + 2 (T_0(o_2)-U_0(o_1))^2 \dot{T}_0(o_2)^2 \\
&+ h_2(o_2) - \dot{T}_0(o_2)^2 + (T_0(o_2)-U_0(o_1))\ddot{T}_0(o_2)\Big] e^{-[T_0(o_2)-U_0(o_1)]^2} dP_0(o_2) dP_0(o_1).
\end{align*}
By the same arguments used to for the second pair of integrals, the above expression is equal to $\Pi(T)$ under $\mathscr{H}_0$. By \ref{it:Spd}, the third pair of integrals can be rewritten as
\begin{align*}
&2\int\int \Big[2(T_0(o_1)-U_0(o_2))\left(D_0^U(o_2) -D_0^T(o_1)\right) - \left(4(T_0(o_1)-U_0(o_2))^2-2\right)D_0^T(o_1)D_0^U(o_2) + 1\Big] \\
&\hspace{4em}\times e^{-[T_0(o_1)-U_0(o_2)]^2} h_1(o_1)h_1(o_2)dP_0(o_2)dP_0(o_1).
\end{align*}
\end{proof}

\begin{proof}[Proof of \autoref{thm:pd2}]
We start by noting that $\left.\frac{1}{2}\frac{d^2}{dt^2}\psi_t\right|_{t=0}$ is equal to
\begin{align*}
\frac{1}{2}&\left[\left.\frac{d^2}{dt^2}\Phi^{TT}(P_t)\right|_{t=0} + \left.\frac{d^2}{dt^2}\Phi^{UU}(P_t)\right|_{t=0} - \left.\frac{d^2}{dt^2}\Phi^{TU}(P_t)\right|_{t=0} - \left.\frac{d^2}{dt^2}\Phi^{UT}(P_t)\right|_{t=0}\right] \\
=&\; \int\int \left[\Gamma_0^{RR}(o_1,o_2) + \Gamma_0^{SS}(o_1,o_2) - \Gamma_0^{RS}(o_1,o_2) - \Gamma_0^{SR}(o_1,o_2)\right] h_1(o_1)h_1(o_2)dP_0(o_2)dP_0(o_1) \\
=& \frac{1}{2}\int\int D_2^{\Psi}(o_1,o_2)h_1(o_1)h_1(o_2)dP_0(o_2)dP_0(o_1),
\end{align*}
where the penultimate equality makes use of \autoref{lem:phi2d}. It is easy to verify that $D_2^{\Psi}(o_1,o_2)=D_2^{\Psi}(o_2,o_1)$ for all $o_1,o_2$. The arguments given below the theorem statement in the main text establish the one-degeneracy of $\Gamma_0$ under $\mathscr{H}_0$ show that $E_{P_0}[D_2^{\Psi}(O,o)]= E_{P_0}[D_2^{\Psi}(o,O)] = 0$ for all $o\in\mathcal{O}$ under $\mathscr{H}_0$. \condref{it:Sderivsbdd} ensures that $\norm{D_2^{\Psi}}_{2,P_0^2}<\infty$, and thus $D_2^{\Psi}$ is $P_0^2$ square integrable and one-degenerate. 

Because the first pathwise derivative is zero under the null, we have that
\begin{align*}
\psi_t - \psi_0&= \frac{1}{2}t^2 \int \int D_2^{\Psi}(o_1,o_2) h(o_1) h(o_2) dP_0(o_1)dP_0(o_2) + o(t^2).
\end{align*}
Thus $D_2^{\Psi}$ is a second-order canonical gradient of $\Psi$ at $P_0$.
\end{proof}

We give a lemma before proving \autoref{thm:remthirdorder}.
\begin{applemma} \label{lem:phi2ndorder}
Fix $P\in\mathscr{M}$. For all $T,U\in\mathscr{S}$, let
\begin{align*}
\Rem_P^{\Phi^{TU}}\triangleq&\; \twonorm{L_P^{TU}}\twonorm{M_P^{TU}} + \onenorm{\Rem_P^T}\onenorm{\Rem_P^U} + \fournorm{M_P^{TU}}^4.
\end{align*}
There exists a mapping $\zeta(P,P_0,\cdot) : \mathscr{S}\rightarrow\mathbb{R}$ such that, for all $T,U\in\mathscr{S}$ for which $T_0(O)\eqd U_0(O)$,
\begin{align*}
\Big|P_0^2\Gamma_P^{TU}-\Phi^{TU}(P_0) - \zeta(P,P_0,T) - \zeta(P,P_0,U)\Big|&\lesssim \Rem_P^{\Phi^{TU}}
\end{align*}
\end{applemma}
\begin{proof}[Proof of \autoref{lem:phi2ndorder}]
In this proof we use $F(P,P_0,T,U)$ to denote any constant which can be written as $\tilde{\zeta}(P,P_0,T)+\tilde{\zeta}(P,P_0,U)$ for expressions $\tilde{\zeta}(P,P_0,T)$ and $\tilde{\zeta}(P,P_0,U)$ which satisfy $\tilde{\zeta}(P,P_0,T)=\tilde{\zeta}(P,P_0,U)$ whenever $T=U$. We will write $c_1 F(P,P_0,T,U) + c_2 F(P,P_0,T,U) = F(P,P_0,T,U)$ for any real numbers $c_1,c_2$. We then fix $\zeta$ to be the final instance of $\tilde{\zeta}$ upon exiting the proof.

Fix $T,U\in\mathscr{S}$. Let $b_0(o_1,o_2)\triangleq T_0(o_1)-U_0(o_2)$ and $b(o_1,o_2)\triangleq T_P(o_1)-U_P(o_2)$ for any $o_1,o_2$. For ease of notation, in the expected values below we will write $B$ and $B_0$ to refer to $b(O_1,O_2)$ and $b_0(O_1,O_2)$, respectively. We also write $T$ for $T_P(O_1)$, $T_0$ for $T_0(O_1)$, $\Rem_P^T$ for $\Rem_P^T(O_1)$, $U$ for $U_P(O_2)$, $U_0$ for $U_0(O_2)$, and $\Rem_P^U$ for $\Rem_P^U(O_2)$.

We have that
\begin{align*}
P_0^2\Gamma_P^{TU}-\Phi^{TU}(P_0)=&\; E_{P_0^2}\left[e^{-B^2} - e^{-B_0^2}\right] + E_{P_0^2}\left[2B\left(D_P^U(O_2) -D_P^T(O_1)\right) e^{-B^2}\right] \\
&-E_{P_0^2}\left[\left(4B^2-2\right)D_P^T(O_1)D_P^U(O_2) e^{-B^2} \right] \\
=&\; E_{P_0^2}\left[e^{-B^2} - e^{-B_0^2}\right] - E_{P_0^2}\left[2B\left(B_0-B\right) e^{-B^2}\right] \\
&+ E_{P_0^2}\left[2B\left(\Rem_P^U -\Rem_P^T\right) e^{-B^2}\right] -E_{P_0^2}\left[\left(4B^2-2\right)\left[T-T_0\right]\left[U-U_0\right] e^{-B^2} \right] \\
&-E_{P_0^2}\left[\left(4B^2-2\right)\left(\left[T-T_0\right]\Rem_P^U + \Rem_P^T\left[U-U_0\right]\right) e^{-B^2} \right] \\
&-E_{P_0^2}\left[\left(4B^2-2\right)\Rem_P^T \Rem_P^U e^{-B^2} \right].
\end{align*}
A third-order Taylor expansion of $b_0\mapsto \exp(-b_0^2)$ about $b_0=b$ yields
\begin{align*}
e^{-b^2}-e^{-b_0^2}=&2b(b_0-b)e^{-b^2} - \left(2b^2-1\right)(b_0-b)^2e^{-b^2} + \frac{2}{3}b\left(2b^2-3\right)(b_0-b)^3e^{-b^2} + O\left((b_0-b)^4\right),
\end{align*}
where the magnitude of the $O((b_0-b)^4)$ term is uniformly bounded above by $C(b_0-b)^4$ for some constant $C>0$ when $b_0$ and $b$ fall in $[-1,1]$. For the second-order term, we have
\begin{align*}
E_{P_0^2}\left[-\left(2B^2-1\right)(B_0-B)^2e^{-B^2}\right]=&\;E_{P_0^2}\left[\left(4B^2-2\right)\left(T-T_0\right)\left(U-U_0\right)e^{-B^2}\right] \\
&-E_{P_0^2}\left[\left(\left[T-T_0\right]^2 + \left[U-U_0\right]^2\right)\left(2B^2-1\right)e^{-B^2}\right].
\end{align*}
Thus we have that
\begin{align}
P_0^2\Gamma_P^{TU}-\Phi^{TU}(P_0) \nonumber=&\, E_{P_0^2}\left[2B\left(\Rem_P^U -\Rem_P^T\right) e^{-B^2}\right] \nonumber \\
&+ O\left(\fournorm{B-B_0}^4\right) - E_{P_0^2}\left[\left(4B^2-2\right)\Rem_P^T \Rem_P^U e^{-B^2} \right] \nonumber \\
&-E_{P_0^2}\left[\left(\left[T-T_0\right]^2 + \left[U-U_0\right]^2\right)\left(2B^2-1\right)e^{-B^2}\right] \nonumber \\
&+ \frac{2}{3}E_{P_0}\left[B\left(2B^2-3\right)(B_0-B)^3e^{-B^2}\right]. \label{eq:almosthere}
\end{align}
A Taylor expansion of $f_1(z)= 2ze^{-z^2}$ shows that there exists a $\tilde{B}_1(o_1,o_2)$ that falls between $B(o_1,o_2)$ and $B_0(o_1,o_2)$ for all $o_1,o_2$ such that
\begin{align}
E_{P_0^2}\left[2B\left(\Rem_P^U -\Rem_P^T\right) e^{-B^2}\right] =&\, E_{P_0^2}\left[\left(\Rem_P^U -\Rem_P^T\right)\left(2B_0e^{-B_0^2} + (B-B_0)\dot{f}_1(\tilde{B})\right)\right] \nonumber \\
=&\, F(P,P_0,T,U) + E_{P_0^2}\left[\left(\Rem_P^U -\Rem_P^T\right)(B-B_0)\dot{f}_1(\tilde{B})\right], \label{eq:taylor1}
\end{align}
where the second equality holds under $\mathscr{H}_0$. The boundedness of $\dot{f}_1$ in $[-2,2]$, the triangle inequality, and the Cauchy-Schwarz inequality yield
\begin{align}
E_{P_0^2}&\left|\left(\Rem_P^U -\Rem_P^T\right)(B-B_0)\dot{f}_1(\tilde{B})\right| \lesssim E_{P_0^2}\left|\left(\Rem_P^U -\Rem_P^T\right)(B-B_0)\right| \nonumber \\
&\lesssim E_{P_0^2}\left|L_P^{TU}(O_1)M_P^{TU}(O_2)\right| + E_{P_0}\left|L_P^{TU}\right|E_{P_0}\left|M_P^{TU}\right| \lesssim \twonorm{L_P^{TU}}\twonorm{M_P^{TU}}. \label{eq:cs}
\end{align}
A Taylor expansion of $f_2(z)= (2z^2-1)e^{-z^2}$ yields that there exists a $\tilde{B}_2$ that falls between $B$ and $B_0$ such that
\begin{align*}
E_{P_0^2}&\left[\left(\left[T-T_0\right]^2 + \left[U-U_0\right]^2\right)\left(2B^2-1\right)e^{-B^2}\right] \\
=&\, E_{P_0^2}\left[\left(\left[T-T_0\right]^2 + \left[U-U_0\right]^2\right)\left(2B_0^2-1\right)e^{-B_0^2}\right] \\
&+ 2E_{P_0^2}\left[\left(\left[T-T_0\right]^2 + \left[U-U_0\right]^2\right)(B-B_0)\left(B(2B^2-3)\right)e^{-B^2}\right] \\
&+ E_{P_0^2}\left[\left(\left[T-T_0\right]^2 + \left[U-U_0\right]^2\right)(B-B_0)^2\frac{\ddot{f}_2(\tilde{B}_2)}{2}\right].
\end{align*}
The first line on the right is equal to $F(P,P_0,T,U)$ under $\mathscr{H}_0$. By the triangle inequality and the boundedness of $\ddot{f}_2$ on $[-2,2]$, the third line satisfies
\begin{align}
E_{P_0^2}&\left[\left(\left[T-T_0\right]^2 + \left[U-U_0\right]^2\right)(B-B_0)^2\frac{\ddot{f}_2(\tilde{B}_2)}{2}\right] \lesssim \sum_{k=0}^4 E_{P_0^2}\left|\left[T-T_0\right]^k \left[U-U_0\right]^{4-k}\right| \nonumber\\
&\lesssim \sum_{k=0}^4 E_{P_0}\left|[M_P^{TU}]^k\right| E_{P_0}\left|[M_P^{TU}]^{4-k}\right| \lesssim \fournorm{M_P^{TU}}^4. \label{eq:fkg}
\end{align}
The final inequality above holds by the FKG inequality \citep{Fortuinetal1971}. It follows that
\begin{align}
E_{P_0^2}&\left[\left(\left[T-T_0\right]^2 + \left[U-U_0\right]^2\right)\left(2B^2-1\right)e^{-B^2}\right] + \frac{2}{3}E_{P_0}\left[B\left(2B^2-3\right)(B_0-B)^3e^{-B^2}\right] \nonumber \\
=&\, \frac{4}{3}E_{P_0^2}\left[\left(\left[T-T_0\right]^3 - \left[U-U_0\right]^3\right)B(2B^2-3)e^{-B^2}\right] + F(P,P_0,T,U) + O(\fournorm{M_P^{TU}}^4) \nonumber \\
=&\, F(P,P_0,T,U) + O(\fournorm{M_P^{TU}}^4), \label{eq:taylor2}
\end{align}
where the final equality holds under $\mathscr{H}_0$ by a Taylor expansion of $z\mapsto z(2z^2-3)e^{-z^2}$ and analogous calculations to those used in (\ref{eq:fkg}). We note that the second equality above uses a different $F$ and a different big-O term than the line above, and that the big-O term can be upper bounded by $C\fournorm{M_P^{TU}}^4$ for a constant $C>0$.

Plugging (\ref{eq:taylor1}), (\ref{eq:cs}), and (\ref{eq:taylor2}) into (\ref{eq:almosthere}), applying the triangle inequality, and using the bounds on $B$ gives the result.
\end{proof}

We give a lemma before proving \autoref{thm:remthirdorder}.
\begin{applemma} \label{lem:Kbd}
Let $K_P\triangleq \onenorm{L_P^{RS}} + \twonorm{M_P^{RS}}^2$ for all $P\in\mathscr{M}$. If $\mathscr{H}_0$ holds, then for all $P\in\mathscr{M}$,
\begin{align*}
\sup_{o_1\in\mathcal{O}'} \left|P_0 \Gamma_P(o_1,\cdot)\right|\lesssim K_P,
\end{align*}
where $\mathcal{O}'\subseteq \mathcal{O}$ is some $P_0$ probability $1$ set. More generally, for all $P_0\in\mathscr{M}$,
\begin{align*}
\left|P_0^2 \Gamma_P-\psi_0\right|&\lesssim K_P.
\end{align*}
\end{applemma}
\begin{proof}[Proof of \autoref{lem:Kbd}]
For $T,U\in\mathscr{S}$, we have that
\begin{align*}
\Gamma_P^{TU}=&\, \left[1+ 2(T_P-U_P) D_P^U\right]e^{-[T_P-U_P]^2} -2\left[(T_P-U_P) + \left(2(T_P-U_P)^2-1\right)D_P^U\right]D_P^Te^{-[T_P-U_P]^2}.
\end{align*}
Above we have omitted the dependence of $\Gamma^{TU}$ on $(o_1,o_2)$, $T$ and $D_P^T$ on $o_1$, and $U$ and $D_P^U$ on $o_2$. For $P_0$ almost all $o_1\in\mathcal{O}$, $P_0\Gamma_P^{TU}(o_1,\cdot)$ is equal to
\begin{align*}
P_0&\left[1+ 2(T_P(o_1)-U_P)(U_P-U_0)\right]e^{-[T_P(o_1)-U_P]^2} + O\left(\onenorm{\Rem_P^U}\right) \\
&-2P_0\left[(T_P(o_1)-U_P) + \left(2(T_P(o_1)-U_P)^2-1\right)(U_P-U_0)\right]D_P^T(o_1)e^{-[T_P(o_1)-U_P]^2}
\end{align*}
where the magnitude of the big-O remainder term is upper bounded by $C\onenorm{\Rem_P^U}$ for a constant $C>0$ which does not depend on $o_1$. Taylor expansions of the first and third terms above yield
\begin{align*}
P_0\Gamma_P^{TU}(o_1,\cdot)=&\, P_0 e^{-[T_P(o_1)-U_0]^2} -2P_0(T_P(o_1)-U_0)D_P^T(o_1)e^{-[T_P(o_1)-U_0]^2} \\
&+ O\left(\onenorm{\Rem_P^U}\right) + O\left(\twonorm{U_P-U_0}^2\right),
\end{align*}
where the magnitude of the big-O term can be upper bounded by $C\twonorm{U_P-U_0}^2$. If $T_0(O)\eqd U_0(O)$, then
\begin{align*}
P_0\Gamma_P^{TU}(o_1,\cdot)=&\, P_0 e^{-[T_P(o_1)-T_0]^2} -2P_0(T_P(o_1)-T_0)D_P^T(o_1)e^{-[T_P(o_1)-T_0]^2} \\
&+ O\left(\onenorm{\Rem_P^U}\right) + O\left(\twonorm{U_P-U_0}^2\right).
\end{align*}
Recall that $T,U\in\mathscr{S}$ were arbitrary. Using that $\Gamma_P\triangleq \Gamma_P^{RR}-\Gamma_P^{RS}-\Gamma_P^{SR}+\Gamma_P^{SS}$ and applying the triangle inequality gives the first result.

We now turn to the second result. For any $T,U\in\mathscr{S}$ and $P\in\mathscr{M}$, we have that
\begin{align*}
P_0^2 \Gamma_P^{TU}=& \Big[2(T_P-U_P)\left(U_0-U_P -T_0 + T_P\right) + 1 \\
&\;- \left(4(T_P-U_P)^2-2\right)(U_P-U_0)(T_P-T_0)\Big]e^{-[T_P-U_P]^2} + O\left(\onenorm{L_P^{TU}}\right) \\
=& \left[2(T_P-U_P)\left(U_0-U_P -T_0 + T_P\right) + 1\right]e^{-[T_P-U_P]^2} + O\left(\onenorm{L_P^{TU}}\right) + O\left(\twonorm{M_P^{TU}}^2\right) \\
=& \Phi^{TU}(P_0) + O\left(\onenorm{L_P^{TU}}\right) + O\left(\twonorm{M_P^{TU}}^2\right),
\end{align*}
where the final equality holds by a first-order Taylor expansion of $(t,u)\mapsto e^{-[t-u]^2}$. The fact that $\Gamma_P\triangleq\Gamma_P^{RR} - 2\Gamma_P^{RS} + \Gamma_P^{SS}$ yields the result.
\end{proof}

\begin{proof}[Proof of \autoref{thm:remthirdorder}]
Fix $P\in\mathscr{M}$ and let $P_0$ satisfy $\mathscr{H}_0$. We have that
\begin{align*}
P_0^2\Gamma_P-\psi_0=&\, P_0^2 \Gamma_P^{RR} - \Phi^{RR}(P_0) + P_0^2 \Gamma_P^{SS} - \Phi^{SS}(P_0) - \left[P_0^2 \Gamma_P^{RS} - \Phi^{RS}(P_0)+P_0^2 \Gamma_P^{SR} - \Phi^{SR}(P_0)\right].
\end{align*}
Taking the absolute value of both sides, applying the triangle inequality, and using \autoref{lem:phi2ndorder} yields
\begin{align*}
\left|P_0^2\Gamma_P-\psi_0\right|&\lesssim \Rem_P^{\Phi^{RR}} + \Rem_P^{\Phi^{SS}} + 2\Rem_P^{\Phi^{RS}}\lesssim \onenorm{L_P^{RS}}^2 + \fournorm{M_P^{RS}}^4 + \twonorm{L_P^{RS}} \twonorm{M_P^{RS}},
\end{align*}
where the final inequality uses the maximum in the definition of $L_P^{RS}$ and $M_P^{RS}$.

The inequality for when $P_0$ satisfies $\mathscr{H}_1$ is proven in \autoref{lem:Kbd}.
\end{proof}

\subsection{Proofs for \secref{sec:est}}
\begin{proof}[Proof of \autoref{lem:firstordnegl}]
By the first result of \autoref{lem:Kbd}, $\left|P_0 \Gamma_n(o_1,\cdot)\right|\lesssim K_n$ for $P_0$ almost all $o_1\in\mathcal{O}'$.  We have that
\begin{align*}
\left|(P_n-P_0)P_0 \Gamma_n\right|&= K_n \left|(P_n-P_0)\left(\frac{P_0 \Gamma_n}{K_n}\right)\right|.
\end{align*}
The fact that $\left\{o_1\mapsto \frac{P_0 \Gamma_n(o_1,\cdot)}{K_n} : \hat{P}_n\right\}$ belongs to a $P_0$ Donsker class with probability approaching $1$, where $\hat{P}_n$ varies over the set of its possible realizations, yields that $(P_n-P_0)\left(\frac{P_0 \Gamma_n}{K_n}\right)=O_{P_0}(n^{-1/2})$ \citep{vanderVaartWellner1996}, and thus the right-hand side above is $O_{P_0}(K_n/\sqrt{n})$. If $K_n=o_{P_0}(n^{-1/2})$, then this yields that the right-hand side above is $o_{P_0}(n^{-1})$.
\end{proof}

\begin{proof}[Proof of \autoref{thm:asympt}]
Plugging \ref{it:1ordempproc}, \ref{it:2ordempproc}, and \ref{it:remnegl} into (\ref{eq:nullexpansion}) yields
\begin{align}
\psi_n-\psi_0&= \Un \Gamma_0 + o_{P_0}(n^{-1}). \label{eq:fixedVstat}
\end{align}
By Section 5.5.2 of \cite{Serfling1980} and the fact that $\Gamma_0$ is $P_0$ degenerate and uniformly bounded, $n\Un \Gamma_0\rightsquigarrow \sum_{k=1}^\infty \lambda_k (Z_k^2-1)$.

We now prove that all of the eigenvalues of $h(o)\mapsto E_{P_0}\left[\tilde{\Gamma}_0(O,o)h(O)\right]$ are nonnegative. Consider a submodel $\{P_t : t\}$ with first-order score $h_1\in L^2(P_0)$ and second-order score $h_2\equiv 0$. By the second-order pathwise differentiability of $\Psi$,
\begin{align*}
\frac{\psi_t-\psi_0}{t^2}&= \frac{1}{2} \int \int D_2^{\Psi}(o_1,o_2)h_1(o_1)h_1(o_2)dP_0(o_1)dP_0(o_2) + o(1).
\end{align*}
The left-hand side is nonnegative for all $t\not=0$ since $\psi_t\ge 0=\psi_0$ under $\mathscr{H}_0$. Thus taking the limit inferior as $t\rightarrow 0$ of both sides shows that
\begin{align*}
\frac{1}{2} \int \int D_2^{\Psi}(o_1,o_2)h_1(o_1)h_1(o_2)dP_0(o_1)dP_0(o_2)&\ge 0.
\end{align*}
Using that $\tilde{\Gamma}_0=\Gamma_0$ under $\mathscr{H}_0$ and $\Gamma_0=\frac{1}{2}D_2^{\Psi}$, we have that $\langle o\mapsto E_{P_0}[\tilde{\Gamma}_0(O,o)h_1(O)],h_1 \rangle\ge 0$, where the inner product is that of $L^2(P_0)$. For any $h_1\in L^2(P_0)$, it is well known that one can choose a submodel $P_t$ with first-order score $h_1\in L^2(P_0)$. Hence the above relation holds for all $h_1\in L^2(P_0)$ and all of the eigenvalues of $h(o)\mapsto E_{P_0}\left[\tilde{\Gamma}_0(O,o)h(O)\right]$ are nonnegative.
\end{proof}

\begin{proof}[Proof of \autoref{cor:Seq0}]
In this case $\Gamma_0(o_1,o_2)=2D_0^R(o_1)D_0^R(o_2)$ under $\mathscr{H}_0$. The central limit theorem yields that $\sigma_1^{-1}\sqrt{n} (P_n-P_0) D_0^R\rightsquigarrow Z$. By the continuous mapping theorem, $\sigma_1^{-2}n(P_n-P_0)^2 \Gamma_0/2\rightsquigarrow Z^2$. Now use that
\begin{align*}
\frac{n \Un \Gamma_0}{2\sigma_1^2}&= \frac{n}{2\sigma_1^2(n-1)}\left[n(P_n-P_0)^2 \Gamma_0 - \frac{1}{n}\sum_{i=1}^n \Gamma_0(O_i,O_i)\right] \\
&= \frac{n}{2\sigma_1^2(n-1)}\left[n(P_n-P_0)^2 \Gamma_0 - \frac{2}{n}\sum_{i=1}^n D_0^R(O_i)^2\right].
\end{align*}
The above quantity converges in distribution to $Z^2 - 1$ by the weak law of large numbers and Slutsky's theorem.
\end{proof}


\begin{proof}[Proof of \autoref{thm:altasdist}]
We have
\begin{align*}
\psi_n&= 2(P_n-P_0) P_0 \Gamma_n + P_0^2 \Gamma_n + \Un \tilde{\Gamma}_n \\
&= 2(P_n-P_0) P_0 \Gamma_0 + P_0^2 \Gamma_n + \Un \tilde{\Gamma}_n + 2(P_n-P_0) P_0\left(\Gamma_n-\Gamma_0\right).
\end{align*}
By assumption, $\Un \tilde{\Gamma}_n=o_{P_0}(n^{-1/2})$. The final term is $o_{P_0}(n^{-1/2})$ by the Donsker condition and the consistency condition \citep{vanderVaartWellner1996}. By the second result of \autoref{lem:Kbd} and the assumption that $K_n=o_{P_0}(n^{-1/2})$, this yields that
\begin{align*}
\psi_n-\psi_0=&\, 2(P_n-P_0) P_0 \Gamma_0 + o_{P_0}(n^{-1/2}).
\end{align*}
Multiplying both sides by $\sqrt{n}$, and applying the central limit theorem yields the result.
\end{proof}

\begin{proof}[Proof of \autoref{cor:testconsistent}]
We have that
\begin{align*}
P_0^n\left\{n\psi_n\le \hat{q}_{1-\alpha}^{ub}\right\}&= P_0^n\left\{\frac{\sqrt{n}(\psi_n-\psi_0)}{\sigma_0}\le \frac{\hat{q}_{1-\alpha}^{ub}n^{-1/2} - \sqrt{n}\psi_0}{\sigma_0}\right\}
\end{align*}
Fix $0<\epsilon<\psi_0$. The right-hand side is equal to
\begin{align*}
P_0^n&\left\{\frac{\sqrt{n}(\psi_n-\psi_0)}{\sigma_0}\le \frac{\hat{q}_{1-\alpha}^{ub}n^{-1/2} - \sqrt{n}\psi_0}{\sigma_0}\textnormal{ and }\hat{q}_{1-\alpha}^{ub}n^{-1}\le \epsilon\right\} + o(1) \\
&\le P_0^n\left\{\frac{\sqrt{n}(\psi_n-\psi_0)}{\sigma_0}\le \frac{\sqrt{n}(\epsilon-\psi_0)}{\sigma_0}\textnormal{ and }\hat{q}_{1-\alpha}^{ub}n^{-1}\le \epsilon\right\} + o(1) \\
&\le P_0^n\left\{\frac{\sqrt{n}(\psi_n-\psi_0)}{\sigma_0}\le \frac{\sqrt{n}(\epsilon-\psi_0)}{\sigma_0}\right\} + o(1) = Pr\left\{Z\le \frac{\sqrt{n}(\epsilon-\psi_0)}{\sigma_0}\right\} + o(1),
\end{align*}
where $Z\sim N(0,1)$. The final equality holds by \autoref{thm:altasdist} and the well known result about the uniform convergence of distribution functions at continuity points when random variables converge in distribution \citep[see, e.g., Theorem 5.6 in][]{Boos&Stefanski2013}. The result follows by noting that $(\epsilon-\psi_0)/\sigma_0$ is negative and that $\lim_{z\rightarrow-\infty}Pr(Z\le z)=0$.
\end{proof}

\bibliography{persrule}

\end{document}